\documentclass[a4paper, 11pt, leqno]{article}

\usepackage{amsmath}
\usepackage{amssymb}
\usepackage{amsthm}
\usepackage{authblk}
\usepackage{graphicx}
\usepackage{psfrag}

\usepackage[paper=a4paper,dvips,top=3.0cm,left=3.0cm,right=3.0cm,foot=1cm,bottom=3.0cm]{geometry}

\usepackage[font=small,labelfont=bf]{caption}

\usepackage{hyperref}
\hypersetup{colorlinks,citecolor=blue,filecolor=black,linkcolor=blue,urlcolor=black}

\newcommand{\abs}[1]{\lvert {#1} \rvert}
\newcommand{\bigabs}[1]{\big\lvert {#1} \big\rvert}
\newcommand{\meas}[2][]{\lvert {#2} \rvert_{#1}}
\newcommand{\closure}[1]{ \, \overline{\!{#1}\!} \, }
\newcommand{\norm}[1]{\Vert {#1} \Vert}
\newcommand{\bignorm}[1]{\big\Vert {#1} \big\Vert}
\newcommand{\seminorm}[1]{\vert {#1} \vert}
\newcommand{\dgnorm}[1]{\vert\hspace{-1pt}\vert\hspace{-1pt}\vert {#1} \vert\hspace{-1pt}\vert\hspace{-1pt}\vert}
\newcommand{\avg}[1]{\{ \hspace{-3.5pt} \{ {#1} \} \hspace{-3.5pt} \}}
\newcommand{\jump}[1]{[ \! [ {#1} ] \! ]}
\def\ud{\, \mathrm{d}}
\def\card{\mathrm{card}}
\def\diam{\mathrm{diam}}
\def\rmD{\mathrm{D}}
\def\rmN{\mathrm{N}}
\def\rmT{\mathrm{T}}
\def\ba{\mathbf{a}}
\def\bn{\mathbf{n}}
\def\bq{\mathbf{q}}
\def\br{\mathbf{r}}
\def\bv{\mathbf{v}}
\def\bw{\mathbf{w}}
\def\bx{\mathbf{x}}
\def\bA{\mathbf{A}}
\def\mF{\mathcal{F}}

\def\mT{\mathcal{T}}

\newtheorem{theorem}{Theorem}[section]
\newtheorem{lemma}[theorem]{Lemma}
\newtheorem{corollary}[theorem]{Corollary}
\newtheorem{assumption}[theorem]{Assumption}
\newtheorem{remark}[theorem]{Remark}

\begin{document}

\title{An interior penalty discontinuous {G}alerkin method for a class of monotone quasilinear elliptic problems}

\author{Peter W. Fick\footnote{Email: p.w.fick@tudelft.nl}}
\affil{\normalsize Faculty of Aerospace Engineering, Delft University of Technology, Delft The Netherlands}

\date{\normalsize January 1, 2014}

\maketitle

\begin{abstract}
A family of interior penalty $hp$-discontinuous Galerkin methods is developed and analyzed for the numerical solution of the quasilinear elliptic equation $-\nabla{} \cdot (\bA(\nabla{u}) \nabla{u} = f$ posed on the open bounded domain $\Omega \subset \mathbb{R}^d$, $d \geq 2$. Subject to the assumption that the map $\bv \mapsto \bA(\bv) \bv$, $\bv \in \mathbb{R}^d$, is Lipschitz continuous and strongly monotone, it is proved that the proposed method is well-posed. \emph{A priori} error estimates are presented of the error in the broken $H^1(\Omega)$-norm, exhibiting precisely the same $h$-optimal and mildly $p$-suboptimal convergence rates as obtained for the interior penalty approximation of linear elliptic problems. \emph{A priori} estimates for linear functionals of the error and the $L^2(\Omega)$-norm of the error are also established and shown to be $h$-optimal for a particular member of the proposed family of methods. The analysis is completed under fairly weak conditions on the approximation space, allowing for non-affine and curved elements with multilevel hanging nodes. The theoretical results are verified by numerical experiments. \\

\end{abstract}


\section{Introduction} \label{section:introduction}

Over the past two decades, discontinuous Galerkin (DG) finite element methods have emerged as an effective and popular choice for the numerical solution of a wide range of partial differential equations. This is mainly stimulated by their high degree of locality, their extreme flexibility with respect to $hp$-adaptive mesh refinement, and their natural ability to accommodate high-order discretizations for hyperbolic problems in a locally conservative manner without excessive numerical stabilization. As it stands, there exists a vast amount of literature on the \emph{a~priori} error analysis of DG methods for linear problems; we refer to the recent book of Di~Pietro \& Ern~\cite{PietroErn2011} for a comprehensive overview of the most prominent results. For nonlinear problems, however, there are still relatively few results available; we mention the works of Houston \emph{et al.}~\cite{HoustonEtAl2005}, Ortner \& S\"{u}li~\cite{OrtnerSuli2007}, Gudi \& Pani~\cite{GudiPani2007}, Gudi \emph{et al.}~\cite{GudiEtAl2008a, GudiEtAl2008b}, \cite{DolejsiEtAl2005}, Dolej\v{s}\'{\i}~\cite{Dolejsi2008}, Bustinza \& Gatica~\cite{BustinzaGatica2004}, and Bi \& Lin~\cite{BiLin2012}. It is fair to say that the extension of DG methods from linear to nonlinear problems is non-obvious in many cases, particularly with respect to the proper formulation of the element boundary terms, and that the analysis turns out to be more challenging.

In this article, we present and analyze a family of interior penalty DG methods for the numerical solution of the following class of quasilinear elliptic boundary value problems. Let~$\Omega$ be an open bounded domain in~$\mathbb{R}^d$, $d \geq 2$, with Lipschitz boundary $\partial \Omega = \Gamma_\rmD \cup \Gamma_\rmN$, where $\Gamma_\rmD \neq \emptyset$ and $\Gamma_\rmN = \partial \Omega \setminus \Gamma_\rmD$. Denoting by~$\bn \colon \Gamma_\rmN \to \mathbb{R}^d$ the unit outward normal to~$\Gamma_\rmN$, our model problem of interest is stated as follows: find $u \colon \closure{\Omega} \to \mathbb{R}$ such that
\begin{subequations} \label{eq:model_problem}
\begin{alignat}{2}
- \nabla \cdot \left( \bA(\bx, \nabla{u}) \nabla{u} \right) = \ & f & \quad & \text{in} \ \Omega , \label{eq:model_problem_pde}
\\
u = \ & g_\rmD & \qquad & \text{on} \ \Gamma_\rmD , \label{eq:model_problem_bcD}
\\
\bA(\bx, \nabla{u}) \nabla{u} \cdot \bn = \ & g_\rmN & \quad & \text{on} \ \Gamma_\rmN , \label{eq:model_problem_bcN}
\end{alignat} 
\end{subequations}
where $\bA \in [L^\infty(\closure{\Omega} \times \mathbb{R}^d)]^{d, d}$, $f \in L^2(\Omega)$, $g_\rmD \in H^{1/2}(\Gamma_\rmD)$ and $g_\rmN \in L^2(\Gamma_\rmN)$. In what follows, we assume that, for~$\bx \in \closure{\Omega}$ and~$\bv \in \mathbb{R}^d$, the nonlinear map $\bv \mapsto \bA(\bx, \bv) \bv$ is \emph{Lipschitz continuous} and \emph{strongly monotone}, as phrased by the following statement.
\begin{assumption} \label{asm:A}
There exist constants $C_\bA \geq M_\bA > 0$ such that, for all~$\bx \in \closure{\Omega}$ and all~$\bv_1, \bv_2 \in \mathbb{R}^d$,
\begin{gather}
\abs{ \bA(\bx, \bv_1) \bv_1 - \bA(\bx, \bv_2) \bv_2} \leq C_\bA \, \abs{\bv_1 - \bv_2} , \label{eq:asm_A1}
\\[3pt]
( \bA(\bx, \bv_1) \bv_1 - \bA(\bx, \bv_2) \bv_2 ) \cdot (\bv_1 - \bv_2) \geq M_\bA \, \abs{\bv_1 - \bv_2}^2 . \label{eq:asm_A2}
\end{gather}
\end{assumption}
Subject to the above assumpion, one can show that problem~\eqref{eq:model_problem} admits a unique weak solution~$u \in H^1(\Omega)$. In passing, we note that problems of the type~\eqref{eq:model_problem} satisfying Assumption~\ref{asm:A} arise in several applications. A classic example is mean curvature flow, for which $\bA(\bx, \nabla{u}) = (1 + \abs{\nabla{u}}^2 )^{-1/2} \, \mathbf{I}$ with~$\mathbf{I}$ the $d \times d$~identity matrix; this has applications in image processing and interface modeling in two-fluid flows, among others. Another example is the modeling of non-Newtonian fluids. For the sake of notational simplicity, we henceforth suppress the dependence of $\bA(\bx, \bv)$ on~$\bx$ and simply write $\bA(\bv)$ instead. 

The development of DG methods for problems of the type~\eqref{eq:model_problem} has also been pursued by several other researchers. In \cite{BustinzaGatica2004}, an $h$-version local DG method is developed and analyzed exhibiting optimal error estimates in the broken $H^1(\Omega)$-norm and $L^2(\Omega)$-norm. The development and analysis of $hp$-version interior penalty DG methods is initiated by Houston \emph{et al.}~\cite{HoustonEtAl2005}. Quasi-optimal error estimates are presented for the error in the broken $H^1(\Omega)$-norm, which are optimal in the mesh size $h$ and mildly supoptimal in the polynomial degree $p$, by half an order in $p$. Estimates for the error in the $L^2(\Omega)$-norm are not presented, but numerical experiments reveal the convergence in the $L^2(\Omega)$-norm to be suboptimal. This suboptimality is caused by so-called dual inconsistency of the method due to a particular formulation of the element boundary terms. Difficulties with respect to the proper formulation of the element boundary terms have motivated other researchers to consider the development of \emph{incomplete} interior penalty DG methods; cf. \cite{OrtnerSuli2007, Dolejsi2008, BiLin2012}. In \cite{GudiEtAl2008b}, a family of interior penalty DG methods is presented and analyzed with a particular choice of the element boundary terms, for which quasi-optimal $hp$-error estimates are derived in both the broken $H^1(\Omega)$-norm and $L^2(\Omega)$-norm.

The purpose of this article is to present and analyze a new family of interior penalty $hp$-DG methods for the numerical solution of~\eqref{eq:model_problem} with quasi-optimal $hp$-error estimates in both the broken $H^1(\Omega)$-norm and $L^2(\Omega)$-norm. As in \cite{HoustonEtAl2005} and \cite{GudiEtAl2008b}, our family of methods depends on the parameter $\theta \in [-1, 1]$. In the linear setting of $\bA(\cdot) = \mathbf{I}$ with $\mathbf{I}$ the $d \times d$ identity matrix and for particular choices of $\theta$, the proposed DG formulation reduces to various well-known interior penalty methods; notable examples include the symmetric and nonsymmetric interior penalty methods of, respectively, Arnold~\cite{Arnold1982} and Rivi\`{e}re \emph{et al.}~\cite{RiviereEtAl1999}. Subject to Assumption~\ref{asm:A}, we prove that the proposed DG formulation is well-posed provided the discontinuity penalization parameter is chosen sufficiently large. Moreover, \emph{a priori} error estimates are presented for the error in the broken $H^1(\Omega)$-norm, displaying precisely the same $h$-optimal and $p$-suboptimal convergence rates as obtained for the interior penalty approximation of linear elliptic problems; cf. \cite{HoustonEtAl2002}. \emph{A priori} estimates for linear functionals of the error and the error in the $L^2(\Omega)$-norm are also derived and shown to be $h$-optimal when $\theta = -1$. The analysis is completed under fairly weak conditions on the $hp$-finite element space allowing for non-affine and curved elements with multilevel hanging nodes and non-uniform polynomial degree.

The remainder of this article is organized as follows. Section~\ref{section:preliminaries} establishes notation, definitions and some auxiliary results. In Section~\ref{section:dgfem}, we introduce the interior penalty $hp$-DG approximation of \eqref{eq:model_problem} and prove several fundamental properties including a well-posedness result. Section~\ref{section:error_analysis} is concerned with the error analysis. Finally, in Section~\ref{section:numerical_experiments} some numerical experiments are presented to illustrate the theoretical results. The appendix is devoted to some auxiliary results regarding the well-posedness of nonlinear variational problems.


\section{Preliminaries} \label{section:preliminaries}

For~$h > 0$, let~$\mT_h$ be a subdivision of~$\Omega$ into disjoint open element domains~$K$ such that $\closure{\Omega} = \cup_{K \in \mT_h} \closure{K}$. Here, $h = \max_{K \in \mT_h} h_K$, where $h_K = \diam(K)$. Each~$K \in \mT_h$ is the image of a fixed reference domain~$\hat{K}$ under a bijective mapping~$T_K \colon \hat{K} \to K$ (that is, $K = T_K(\hat{K})$ for all~$K \in \mT_h$), where~$\hat{K}$ is either the open unit simplex or the open unit hypercube in~$\mathbb{R}^d$. For~$K \in \mT_h$, we denote by~$\bn_K$ the unit outward normal with respect to~$\partial K$. Furthermore, for any pair of neighboring elements~$K, K' \in \mT_h$, we refer to the nonempty $(d-1)$-dimensional interior of~$\partial K \cap \partial K'$ as an interior face of~$\mT_h$. Likewise, for any~$K \in \mT_h$, a boundary face lying on~$\Gamma_\rmD$ (resp.~$\Gamma_\rmN$) is the nonempty $(d-1)$-dimensional interior of~$\partial K \cap \Gamma_\rmD$ (resp.~$\partial K \cap \Gamma_\rmN$). The interior faces and the boundary faces lying on~$\Gamma_\rmD$ and~$\Gamma_\rmN$ are collected in the sets~$\mF_{h, 0}$, $\mF_{h, \rmD}$ and~$\mF_{h, \rmN}$, respectively, and we define $\mF_h := \mF_{h, 0} \cup \mF_{h, \rmD} \cup \mF_{h, \rmN}$. In addition, we let $\mF_{h, 0, \rmD} := \mF_{h, 0} \cup \mF_{h, \rmD}$, and, for each~$K \in \mT_h$, we denote by~$\mF_{h, K}$ the set of faces lying on~$\partial K$; i.e., $\mF_{h, K} := \{ F \in \mF_h \, : \, F \subset \partial K \}$. The union of all interior faces is denoted by~$\Gamma_{h, 0}$ (i.e., $\Gamma_{h, 0} := \cup_{F \in \mF_{h, 0}} F$), and analogously we let~$\Gamma_{h, \rmD}$ and~$\Gamma_{h, \rmN}$ represent the union of faces lying on~$\Gamma_\rmD$ and~$\Gamma_\rmN$. We also define $\Gamma_{h, 0, \rmD} := \Gamma_{h, 0} \cup \Gamma_{h, \rmD}$. 

To characterize functions on~$\mT_h$ that are possibly discontinuous across inter-element boundaries, we introduce the broken Sobolev space
\begin{equation*}
H^s(\Omega, \mT_h) := \{ v \in L^2(\Omega) : \left. v \right|_K \in H^s(K) , \ \forall K \in \mT_h \} ,
\end{equation*}
where $0 < s \leq \infty$. Here, $H^s(K)$ denotes the standard Sobolev-Slobodeckij space of order~$s$ for the domain~$K \in \mT_h$. The space~$H^s(\Omega, \mT_h)$ is equipped with the broken norm and semi-norm
\begin{equation*}
\norm{v}_{H^s(\Omega, \mT_h)} := \left( \sum_{K \in \mT_h} \norm{v}_{H^s(K)}^2 \right)^{1/2} , 
\quad 
\seminorm{v}_{H^s(\Omega, \mT_h)} := \left( \sum_{K \in \mT_h} \seminorm{v}_{H^s(K)}^2 \right)^{1/2} ,
\end{equation*}
where $\norm{\cdot}_{H^s(K)}$ and~$\seminorm{\cdot}_{H^s(K)}$ denote the standard Sobolev-Slobodeckij norm and semi-norm, respectively.

Next, we define jump and average operators for scalar- and vector-valued functions. Let~$K, K' \in \mT_h$ be two adjacent element domains sharing an interior face~$F \in \mathcal{F}_{h, 0}$. Given a scalar-valued function~$v \in H^1(\Omega, \mT_h)$, we define the jump and average of~$v$ at~$F$ by
\begin{equation*}
\left. \jump{v} \right|_F := \left. v \right|_{K} \bn_{K} + \left. v \right|_{K'} \bn_{K'} , 
\qquad
\left. \avg{v} \right|_F := (\left. v \right|_{K} + \left. v \right|_{K'}) / 2 . 
\end{equation*}
Analogously, for a vector-valued function~$\bq \in [H^1(\Omega, \mT_h)]^d$, we set
\begin{equation*}
\left. \jump{\bq} \right|_F := \left. \bq \right|_{K} \cdot \bn_{K} + \left. \bq \right|_{K'} \cdot \bn_{K'} , 
\qquad
\left. \avg{\bq} \right|_F := (\left. \bq \right|_{K} + \left. \bq \right|_{K'}) / 2 .
\end{equation*}
If~$F \in \mF_{h, \rmD}$ or~$F \in \mF_{h, \rmN}$, we moreover define $\left. \jump{v} \right|_F := \left. v \right|_K \bn_K$, $\left. \avg{v} \right|_F := \left. v \right|_K$ and~$\left. \avg{\bq} \right|_F := \left. \bq \right|_K$, where~$K \in \mT_h$ such that~$F \subset \partial K$; the quantity~$\left. \jump{\bq} \right|_F$ is not required for~$F \in \mF_{h, \rmD} \cup \mF_{h, \rmN}$ and is thus left undefined.

Given a nonnegative integer~$k$, let~$\hat{P}_k(\hat{K})$ denote the space of polynomials of total degree up to~$k$ with support on the reference domain~$\hat{K}$. Also, let~$\hat{Q}_k(\hat{K})$ denote the space of tensor-product polynomials of degree up to~$k$ in each coordinate direction of~$\hat{K}$. We define~$\hat{S}_k(\hat{K}) = \hat{P}_k(\hat{K})$ when~$\hat{K}$ is the unit $d$-simplex, and~$\hat{S}_k(\hat{K}) = \hat{Q}_k(\hat{K})$ when~$\hat{K}$ is the unit $d$-hypercube. In addition, let~$S_k(K) = \{ v : v \, \circ \, T_K \in \hat{S}_k(\hat{K}) \}$. Then, assigning to each~$K \in \mT_h$ an integer~$p_K \geq 1$ to represent the local polynomial degree, we introduce the $hp$-finite element space
\begin{equation*}
V_{h, p} = \{ v \in H^1(\Omega, \mT_h) : \left. v \right|_K \in S_{p_K}(K) , \ \forall K \in \mT_h \} ,
\end{equation*}
where~$p = \min_{K \in \mT_h} p_K$. 

In the analysis that follows, we make some structural assumptions on the subdivision~$\mT_h$ and the distribution of the local polynomial degrees~$\{ p_K \}_{K \in \mT_h}$.
\begin{assumption} \label{asm:B} \
\begin{enumerate}
\item[(i)] For each~$K \in \mT_h$ and some integer~$r_K \geq 2$, the map~$T_K \colon \hat{K} \to K$ is a $C^{r_K}$-diffeomorphism satisfying $\seminorm{T_K}_{[W_\infty^{s}(\hat{K})]^{d, d}} \leq \beta_1 \, h_K^s$ and $\seminorm{T_K^{-1}}_{[W_\infty^s(K)]^{d, d}} \leq \beta_1 \, h_K^{-s}$ for~$s \in [0, r_K]$, with constant~$\beta_1$ independent of~$h_K$.
\item[(ii)] The subdivision~$\mT_h$ is \emph{uniformly graded}; i.e., there exists a constant $\beta_2 > 0$ such that, for all pairs of neighboring elements~$K, K' \in \mT_h$ sharing a face~$F \in \mF_{h, 0}$, there holds $\beta_2^{-1} \leq h_K / h_{K'} \leq \beta_2$.
\item[(iii)] The polynomial degrees~$\{ p_K \}_{K \in \mT_h}$ have \emph{bounded local variation}; i.e., there exists a constant~$\beta_3 > 0$ such that, for all pairs of neighboring elements~$K, K' \in \mT_h$ sharing a face~$F \in \mF_{h, 0}$, there holds $\beta_3^{-1} \leq p_K / p_{K'} \leq \beta_3$.
\end{enumerate}
\end{assumption}
Note that we allow for fairly general subdivisions composed of possibly non-affine and curved elements with multilevel hanging nodes. The only requirement is that each~$K \in \mT_h$ is nondegenerate and sufficiently ``close`` to some affine image of the reference domain~$\hat{K}$ (cf. Assumption~\ref{asm:B}(i); see also, for example, \cite{CiarletRaviart1972}), and that the number of hanging nodes per element face is bounded for all~$K \in \mT_h$ (cf. Assumption~\ref{asm:B}(ii)). We remark that, if~$\mT_h$ is composed of affine images of simplices and/or multilinear images of hypercubes, then Assumption~\ref{asm:B}(i) reduces to a standard shape regularity condition. 

We end this section with some auxiliary results that are needed for the subsequent analysis. Here, and in the sequel, we denote by~$C$ and~$C_i$~($i = 1, 2, \dots$) generic constants, possibly different on each occurrence, which are independent of~$h$ and~$p$. In addtion, we write~$C \equiv C( \lambda_1 , \dots, \lambda_N)$ to indicate the dependence of the constant~$C$ on the parameters~$\lambda_1, \dots, \lambda_N$. We state without proof the following trace inequality; the proof is analogous to that of Lemma~1.49 in~\cite{PietroErn2011}.
\begin{lemma}[Multiplicative trace inequality] \label{lem:trace_inequality}
Let $K \in \mT_h$ and $F \in \mF_{h, K}$. Then, for any $v \in H^{s + 1}(K)$, $0 \leq s \leq r_K - 1$, there exists a constant~$C \equiv C(d, \beta_1)$ such that
\begin{equation} \label{eq:trace_inequality}
\norm{v}_{H^s(F)}^2
\leq 
C \left( h_K^{-1} \norm{v}_{H^s(K)}^2 + \norm{v}_{H^s(K)} \, \norm{v}_{H^{s+1}(K)} \right) .
\end{equation}
\end{lemma}
For future reference, we also state the following $hp$-type inverse estimates; cf.~\cite[Lemma~3]{Quarteroni1984}.
\begin{lemma}[Inverse estimates] \label{lem:inv_estimates}
Let~$K \in \mT_h$ and $F \in \mF_{h, K}$, and denote by $\meas[d]{K}$ and $\meas[d-1]{F}$ the corresponding Hausdorff measures of dimension $d$ and $d-1$, respectively. Then, for any $v \in S_{p_K}(K)$, there exists a constant~$C \equiv C(d, \beta_1)$ such that:
\begin{enumerate}
\item[(i)] for~$0 \leq s \leq r_K - 1$,
\begin{equation}
\norm{v}_{H^{s + 1}(K)} 
\leq 
C \, p_K \, h_K^{-1/2} \, \norm{v}_{H^s(K)} \, ; \label{eq:inv_estimate_1}
\end{equation}
\item[(ii)] for~$0 \leq s \leq r_K$,
\begin{align}
\norm{v}_{W_\infty^s(K)}
\leq \ &
C \, p_K \, \meas[d]{K}^{-1/2} \, \norm{v}_{H^s(K)} \, , \label{eq:inv_estimate_2a}
\\[6pt]
\norm{v}_{W_\infty^s(F)}
\leq \ &
C \, p_K \, \meas[d-1]{F}^{-1/2} \, \norm{v}_{H^s(F)} \, . \label{eq:inv_estimate_2b}
\end{align}
\end{enumerate}
\end{lemma}
Using the trace inquality~\eqref{eq:trace_inequality} and the inverse estimate~\eqref{eq:inv_estimate_1}, and taking into consideration Assumption~\ref{asm:B}, we prove the following result.
\begin{lemma} \label{lem:inv_trace_inequality}
Let
\begin{equation}
\mu_F :=
\begin{cases}
\frac{1}{2} ( \meas[d]{K} + \meas[d]{K'} ) \, / \, \meas[d-1]{F} & \text{for} \ F \in \mF_{h, 0},
\\[3pt]
\meas[d]{K} \, / \, \meas[d-1]{F} & \text{for} \ F \in \mF_{h, \rmD} \cup \mF_{h, \rmN} ,
\end{cases} \label{eq:mu_F}
\end{equation}
and
\begin{equation}
p_F := 
\begin{cases}
\frac{1}{2}(p_K + p_{K'}) & \text{for} \ F \in \mF_{h, 0} ,
\\[3pt]
p_K & \text{for} \ F \in \mF_{h, \rmD} \cup \mF_{h, \rmN} ,
\end{cases} \label{eq:p_F}
\end{equation}
where~$K, K' \in \mT_h$ (resp. $K \in \mT_h$) are the element domains adjacent to the face~$F \in \mF_{h, 0}$ (resp. $F \in \mF_{h, \rmD} \cup \mF_{h, \rmN}$). There exists a constant~$C \equiv C(d, \beta_1, \beta_2, \beta_3)$ such that, for all~$v \in V_{h, p}$,
\begin{equation}
\sum_{F \in \mF_h} \frac{\mu_F}{p_F^2} \int_F \avg{\abs{\nabla{v}}}^2 \ud s 
\leq 
C \sum_{K \in \mT_h} \int_K \abs{\nabla{v}}^2 \ud x . \label{eq:inv_trace_inequality}
\end{equation}
\end{lemma}
\begin{proof}
Let~$K \in \mT_h$ and~$F \in \mF_{h, K}$. From Assumption~\ref{asm:B}(i) and~\ref{asm:B}(ii) it follows that there exists a constant~$C_1 \equiv C_1(d, \beta_1, \beta_2)$ such that~$\mu_F \leq C_1 \, h_K$. Moreover, Assumption~\ref{asm:B}(iii) implies that~$p_F^2 \geq C_2 \, p_K^2$ for some positive constant~$C_2 \equiv C_2(\beta_3)$. Hence, by the Young's inequality, we deduce that
\begin{equation*}
\sum_{F \in \mF_h} \frac{\mu_F}{p_F^2} \int_F \avg{\abs{\nabla{v}}}^2 \ud s 
\leq
\frac{C_1}{C_2} \sum_{K \in \mT_h} \frac{h_K}{p_K^2} \sum_{F \in \mF_{h, K}} \int_F \bigabs{\left.(\nabla{v})\right|_K}^2 \ud s .
\end{equation*}
On account of Assumption~\ref{asm:B}(ii) we have that $\mathrm{card}(\mF_{h, K}) \leq C_3$ for some positive integer~$C_3 \equiv C_3(d, \beta_2)$. Using the trace inequality~\eqref{eq:trace_inequality} with constant $C_4 \equiv C_4(d, \beta_1)$, we then obtain:
\begin{equation*}
\sum_{F \in \mF_h} \frac{\mu_F}{p_F^2} \int_F \avg{\abs{\nabla{v}}}^2 \ud s 
\leq
C_3 \, C_4 \, \frac{C_1}{C_2} \sum_{K \in \mT_h} \frac{h_K}{p_K^2} \left( h_K^{-1} \norm{v}_{H^1(K)}^2 + \norm{v}_{H^1(K)} \norm{v}_{H^2(K)} \right) .
\end{equation*}
The proof is concluded by applying the inverse estimate~\eqref{eq:inv_estimate_1}.
\end{proof}


\section{Discontinuous Galerkin method} \label{section:dgfem}

Let us consider the sum space~$V(h, p) := V_{h, p} + H^s(\Omega)$, $s > 3/2$. For~$w, v \in V(h, p)$, we introduce the semilinear form
\begin{equation}
N(w; v) 
= 
\sum_{K \in \mT_h} \bA(\nabla{w}) \nabla{w} \cdot \nabla{v} \ud x
+ 
B_0(w; v) 
+ 
B_\rmD(w; v) , \label{eq:N}
\end{equation}
and the linear form
\begin{equation}
L(v)
=
\sum_{K \in \mT_h} \int_K f v \ud x
+
\int_{\Gamma_{h, \rmN}} g_\rmN \, v \ud s . \label{eq:L}
\end{equation}
Here,
\begin{align*}
B_0(w; v) 
= &
- \int_{\Gamma_{h, 0}} \avg{\bA(\nabla{w} - \sigma \jump{w}) \nabla{w}} \cdot \jump{v} \ud s
\\ &
+
\theta \int_{\Gamma_{h, 0}} \avg{\bA^\rmT(\nabla{w} - \sigma \jump{w}) \nabla{v} } \cdot \jump{w} \ud s
\\ &
+
\int_{\Gamma_{h, 0}} \sigma \avg{\bA(\nabla{w} - \sigma \jump{w})} \, \jump{w} \cdot \jump{v} \ud s
\\ &
+
\theta \int_{\Gamma_{h, 0}} \sigma^{-1} \avg{( \bA(\nabla{w}) - \bA(\nabla{w} - \sigma \jump{w}) ) \nabla{w} \cdot \nabla{v}} \ud s
\end{align*}
and
\begin{align*}
B_\rmD(w; v) 
= &
- \int_{\Gamma_{h, \rmD}} \bA(\nabla{w} - \sigma \bn (w - g_\rmD)) \nabla{w} \cdot \bn v \ud s
\\ &
+
\theta \int_{\Gamma_{h, \rmD}} \bA^\rmT(\nabla{w} - \sigma \bn (w - g_\rmD)) \nabla{v} \cdot \bn (w - g_\rmD) \ud s
\\ &
+
\int_{\Gamma_{h, \rmD}} \bA(\nabla{w} - \sigma \bn (w - g_\rmD)) \bn v \cdot \bn (w - g_\rmD) \ud s
\\ &
+
\theta \int_{\Gamma_{h, \rmD}} \sigma^{-1} ( \bA(\nabla{w}) - \bA(\nabla{w} - \sigma \bn (w - g_\rmD)) ) \nabla{w} \cdot \nabla{v} \ud s ,
\end{align*}
where~$\bA^\rmT(\cdot)$ denotes the transpose of~$\bA(\cdot)$, $\theta$ is a fixed constant in~$[-1, 1]$, and~$\sigma$ is a piecewise constant function on~$\Gamma_{h, 0, \rmD}$, defined by
\begin{equation*}
\left. \sigma \right|_F = \alpha \, \frac{ p_F^2 }{ \mu_F } , \qquad F \in \mF_{h, 0, \rmD} .
\end{equation*}
Here, $\mu_F$ and~$p_F$ are defined as in~\eqref{eq:mu_F} and~\eqref{eq:p_F}, and $\alpha$ is the so-called \emph{interior penalty parameter}, which is a positive constant independent of~$h$ and~$p$. As usual, we require that~$\alpha$ is sufficiently large. Anticipating the result of Theorem~\ref{thm:wellposedness_dgfem}, we state that $\alpha > \alpha_0 = 2 \, C \, (1 + \lambda_\theta \, C_\bA / M_\bA)^2$ will suffice, where~$\lambda_\theta = 1 + \abs{1 + \theta}$ and $C$ is the constant from Lemma~\ref{lem:inv_trace_inequality}.

The interior penalty $hp$-DG approximation of~\eqref{eq:model_problem} is now stated as follows: find~$u_{h, p} \in V_{h, p}$ such that
\begin{equation}
N(u_{h, p}; v) = L(v) \qquad \forall \, v \in V_{h, p} . \label{eq:dgfem}
\end{equation}
We note that, in the linear case of~$\bA(\cdot) = \mathbf{I}$, with~$\mathbf{I}$ the $d \times d$ identity matrix, and for particular choices of the parameters~$\theta$ and~$\alpha$, the DG formulation~\eqref{eq:dgfem} reduces to various well-known DG methods. Notable examples include the \emph{symmetric} interior penalty method for~$\theta = -1$ and~$\alpha > \alpha_0 > 0$ (cf.~\cite{Arnold1982}), and the \emph{nonsymmetric} interior penalty method for~$\theta = 1$ and~$\alpha > 0$ (cf.~\cite{RiviereEtAl1999}).

Under suitable regularity conditions, one can show that~\eqref{eq:dgfem} is a consistent approximation of~\eqref{eq:model_problem}.
\begin{lemma}[Galerkin orthogonality] \label{lem:galerkin_orthogonality}
Assume that~\eqref{eq:model_problem} has a strong solution~$u \in H^s(\Omega) \cap C^0(\Omega)$, $s > 3/2$. Then,
\begin{equation} \label{eq:galerkin_orthogonality}
N(u; v) - N(u_{h, p}; v) = 0 \qquad \forall v \in V_{h, p} .
\end{equation} 
\end{lemma}
\begin{proof}
Since~$u \in C^0(\Omega)$, we have that~$\left. \jump{u} \right|_F = 0$ strongly for all~$F \in \mF_{h, 0}$. Moreover, since~$u$ satisfies~\eqref{eq:model_problem_pde} almost everywhere, we have that~$\nabla \cdot \bA(\nabla{u}) \nabla{u} \in L^2(\Omega)$. From~\cite[Lemma~1.24]{PietroErn2011}, it then follows that $\left. \jump{\bA(\nabla{u}) \nabla{u}} \right|_F = 0$ almost everywhere for all~$F \in \mF_{h, 0}$. Therefore, upon integration by parts, we find that~$N(u; v) = L(v)$ for all~$v \in V_{h, p}$, from which we infer the stated result.
\end{proof}

For the analysis of the $hp$-DG approximation~\eqref{eq:dgfem}, we introduce the norms
\begin{gather*}
\dgnorm{v}^2
:= 
\sum_{K \in \mT_h} \int_K \abs{\nabla{v}}^2 \ud x 
+ 
\int_{\Gamma_{h, 0, \rmD}} \sigma \abs{\jump{v}}^2 \ud s , \qquad v \in V(h, p) ,
\\
\dgnorm{v}_+^2
:= 
\dgnorm{v}^2
+ 
\int_{\Gamma_{h, 0, \rmD}} \sigma^{-1} \avg{\abs{\nabla{v}}}^2 \ud s , \qquad v \in V(h, p) .
\end{gather*}
We note that these norms are equivalent on~$V_{h, p}$ for any~$\alpha > 0$. Indeed, by Lemma~\ref{lem:inv_trace_inequality} there exists a constant~$C$ such that
\begin{equation} 
\dgnorm{v}^2 \leq \dgnorm{v}_+^2 \leq (1 + C \alpha^{-1}) \, \dgnorm{v}^2 \qquad \forall v \in V_{h, p} . \label{eq:norm_equivalence}
\end{equation}

Next, let $X(\Gamma_{h, 0, \rmD}) = \Pi_{K \in \mT_h} L^2(\partial K \cap \Gamma_{h, 0, \rmD})$ and define the trace operator $\widehat{\nabla}_\sigma \colon V(h, p) \to [X(\Gamma_{h, 0, \rmD})]^d$ such that, for~$K \in \mT_h$ and~$F \in \mF_{h, K}$,
\begin{equation*}
\left. ( \widehat{\nabla}_\sigma \, w ) \right|_K
=
\begin{cases}
\left. \left( \left. (\nabla{w}) \right|_K \right) \right|_{F} - \sigma \jump{w} & \text{if} \ F \in \mF_{h, 0} , 
\\[3pt]
\left. \left( \left. (\nabla{w}) \right|_K \right) \right|_{F} - \sigma \bn (w|_K - g_\rmD) & \text{if} \ F \in \mF_{h, \rmD} .
\end{cases}
\end{equation*}
By the fact that~$\avg{\jump{\cdot}} = \jump{\cdot}$, we have the following useful identity: 
\begin{equation} \label{eq:N_alt}
\begin{aligned} 
N(w; v) 
= &
\sum_{K \in \mT_h} \bA(\nabla{w}) \nabla{w} \cdot \nabla{v} \ud x
\\ &
-
\int_{\Gamma_{h, 0, \rmD}} \avg{ \bA(\widehat{\nabla}_\sigma \, w) \, \widehat{\nabla}_\sigma \, w \cdot (\theta \sigma^{-1} \nabla{v} + \jump{v}) } \ud s
\\ &
+
\theta \int_{\Gamma_{h, 0, \rmD}} \sigma^{-1} \avg{\bA(\nabla{w}) \nabla{w} \cdot \nabla{v}} \ud s .
\end{aligned}
\end{equation}
Rewriting the semilinear form~$N$ according to~\eqref{eq:N_alt} and using Assumption~\ref{asm:A}, we are able to prove the following two lemmata.
\begin{lemma}[Lipschitz continuity] \label{lem:lipschitz_continuity}
There exists a constant~$C_N \equiv C_N(\theta, C_\bA)$ such that
\begin{equation}
N(w_1; v) - N(w_2; v) \leq C_N \, \dgnorm{w_1 - w_2}_+ \, \dgnorm{v}_+ \qquad \forall w_1, w_2, v \in V(h, p) . \label{eq:lipschitz_continuity}
\end{equation}
\end{lemma}
\begin{proof}
Starting from~\eqref{eq:N_alt} and using that~$\abs{\avg{\bq_1 \cdot \bq_2}} \leq \avg{\abs{\bq_1} \ \abs{\bq_2}} \leq 2 \avg{\abs{\bq_1}} \, \avg{\abs{\bq_2}}$ for all~$\bq_1, \bq_2 \in [H^1(\Omega, \mT_h)]^d$, we have that
\begin{align*}
& N(w_1; v) - N(w_2; v) 
\leq
\sum_{K \in \mT_h} \int_K \abs{\bA(\nabla{w_1}) \nabla{w_1} - \bA(\nabla{w_2}) \nabla{w_2}} \ \abs{\nabla{v}} \ud x
\\ & \quad
+ 
\int_{\Gamma_{h, 0, \rmD}} \avg{ \abs{\bA(\widehat{\nabla}_\sigma \, w_1) \, \widehat{\nabla}_\sigma \, w_1 - \bA(\widehat{\nabla}_\sigma \, w_2) \, \widehat{\nabla}_\sigma \, w_2} } \left( 2 \abs{\theta} \sigma^{-1} \avg{\abs{\nabla{v}}} + \abs{\jump{v}} \right) \ud s
\\ & \quad
+
2 \, \abs{\theta} \, \int_{\Gamma_{h, 0, \rmD}} \sigma^{-1} \avg{\abs{\bA(\nabla{w_1}) \nabla{w_1} - \bA(\nabla{w_2}) \nabla{w_2}}} \avg{\abs{\nabla{v}}} \ud s .
\end{align*}
We use the Lipschitz condition~\eqref{eq:asm_A1} from~Assumption~\ref{asm:A} to bound each of these terms, yielding
\begin{align*}
& N(w_1; v) - N(w_2; v)
\leq 
C_\bA \sum_{K \in \mT_h} \int_K \abs{\nabla{w_1}- \nabla{w_2}} \ \abs{\nabla{v}} \ud x
\\ & \quad
+
2 \abs{\theta} \, C_\bA \int_{\Gamma_{h, 0, \rmD}} \abs{\jump{w_1 - w_2}} \, \avg{\abs{\nabla{v}}} \ud s 
+
C_\bA \int_{\Gamma_{h, 0, \rmD}} \sigma \abs{\jump{w_1 - w_2}} \ \abs{\jump{v}} \ud s
\\ & \quad
+ 
4 \abs{\theta} \, C_\bA \int_{\Gamma_{h, 0, \rmD}} \sigma^{-1} \avg{\abs{\nabla{w_1} - \nabla{w_2}}} \, \avg{\abs{\nabla{v}}} \ud s .
\end{align*}
Upon application of the Cauchy-Schwarz inequality, we arrive at~\eqref{eq:lipschitz_continuity} with $C_N = (2 + 4 \abs{\theta}) \, C_\bA$. 
\end{proof}

\begin{lemma}[Strong monotonicity] \label{lem:strong_monotonicity}
Let~$\theta \in [-1, 1]$ and select $\alpha > \alpha_0 = 2 \, C ( 1 + \lambda_\theta \, C_\bA / M_\bA )^2$, where $\lambda_\theta = 1 + \abs{1 + \theta}$ and~$C$ is the constant from Lemma~\ref{lem:inv_trace_inequality}. There exists a positive constant~$M_N \equiv M_N(M_\bA, \alpha_0 / \alpha)$ such that 
\begin{equation}
N(w_1; w_1 - w_2) - N(w_2; w_1 - w_2) \geq M_N \, \dgnorm{w_1 - w_2}^2 \qquad \forall w_1, w_2 \in V_{h, p} . \label{eq:strong_monotonicity}
\end{equation}
\end{lemma}
\begin{proof}
Let us write $w = w_1 - w_2$. Starting from~\eqref{eq:N_alt}, we have that
\begin{equation} \label{eq:N_alt_2}
N(w_1; w_1 - w_2) - N(w_2; w_1 - w_2)
=
T_1 + T_2 + T_3 + T_4 , 
\end{equation}
where
\begin{align*}
T_1 = & \sum_{K \in \mT_h} \int_K ( \bA(\nabla{w_1}) \nabla{w_1} - \bA(\nabla{w_2}) \nabla{w_2} ) \cdot \nabla{w} \ud x ,
\\
T_2 = & \int_{\Gamma_{h, 0, \rmD}} \sigma^{-1} \avg{ ( \bA(\widehat{\nabla}_\sigma w_1) \widehat{\nabla}_\sigma w_1 - \bA(\widehat{\nabla}_\sigma w_2) \widehat{\nabla}_\sigma w_2 ) \cdot \widehat{\nabla}_\sigma w } \ud s ,
\\
T_3 = & - (1 + \theta) \int_{\Gamma_{h, 0, \rmD}} \sigma^{-1} \avg{ ( \bA(\widehat{\nabla}_\sigma w_1) \widehat{\nabla}_\sigma w_1 - \bA(\widehat{\nabla}_\sigma w_2) \widehat{\nabla}_\sigma w_2 ) \cdot \nabla{w} } \ud s ,
\\
T_4 = & \ \theta \int_{\Gamma_{h, 0, \rmD}} \sigma^{-1} \avg{ ( \bA(\nabla{w_1}) \nabla{w_1} - \bA(\nabla{w_2}) \nabla{w_2} ) \cdot \nabla{w} } \ud s .
\end{align*}
Using the monotonicity condition~\eqref{eq:asm_A2} from Assumption~\ref{asm:A}, it immediately follows that
\begin{equation*}
T_1 \geq M_\bA \sum_{K \in \mT_h} \int_K \abs{\nabla{w}}^2 \ud x .
\end{equation*}
Analogously, for~$T_2$, we find that
\begin{align*}
T_2 
\geq \ &
M_\bA \int_{\Gamma_{h, 0, \rmD}} \sigma^{-1} \avg{ \abs{\widehat{\nabla} w}^2 } \ud s
\\ = \ &
M_\bA \int_{\Gamma_{h, 0, \rmD}} \left( \sigma^{-1} \avg{\abs{\nabla{w}}^2} - 2 \avg{\nabla{w}} \cdot \jump{w} + \sigma \abs{\jump{w}}^2 \right) \ud s 
\\ \geq \ &
- 2 M_\bA \int_{\Gamma_{h, 0, \rmD}} \avg{\abs{\nabla{w}}} \ \abs{\jump{w}} \ud s 
+
M_\bA \int_{\Gamma_{h, 0, \rmD}} \sigma \abs{\jump{w}}^2 \ud s .
\end{align*}
The first term on the right hand side can be further bounded by using the Young's inequality $2 a b \leq \epsilon^{-1} a^2 + \epsilon b^2$, where $a = \sigma^{-1/2} \avg{\abs{\nabla{w}}}$, $b = \sigma^{1/2} \abs{\jump{w}}$ and~$\epsilon > 0$. Subsequently applying Lemma~\ref{lem:inv_trace_inequality}, we obtain
\begin{align*}
T_2
\geq &
- M_\bA \, \epsilon^{-1} \int_{\Gamma_{h, 0, \rmD}} \sigma^{-1} \avg{\abs{\nabla{w}}}^2 \ud s 
+
M_\bA (1 - \epsilon) \int_{\Gamma_{h, 0, \rmD}} \sigma \abs{\jump{w}}^2 \ud s
\\ \geq &
- M_\bA \, \epsilon^{-1} \, C \, \alpha^{-1} \sum_{K \in \mT_h} \int_K \abs{\nabla{w}}^2 \ud x
+
M_\bA (1 - \epsilon) \int_{\Gamma_{h, 0, \rmD}} \sigma \abs{\jump{w}}^2 \ud s ,
\end{align*}
where~$C$ is the constant from Lemma~\ref{lem:inv_trace_inequality}. For~$T_3$, using the Lipschitz condition~\eqref{eq:asm_A1} from Assumption~\ref{asm:A} together with the fact that~$\avg{\abs{\cdot}^2} \leq 2 \avg{\abs{\cdot}}^2$, and proceeding similarly as for~$T_2$, we have that
\begin{align*}
T_3 
\geq &
- \abs{1 + \theta} \, C_\bA \int_{\Gamma_{h, 0, \rmD}} \sigma^{-1} \avg{\abs{ \widehat{\nabla}_\sigma w} \ \abs{\nabla{w}}} \ud s
\\ \geq &
- \abs{1 + \theta} \, C_\bA \int_{\Gamma_{h, 0, \rmD}} ( 2 \sigma^{-1} \avg{\abs{\nabla{w}}}^2 + \abs{\jump{w}} \ \avg{\abs{\nabla{w}}} ) \ud s
\\ \geq &
- \abs{1 + \theta} \, C_\bA \left( (2 + \epsilon^{-1}) \int_{\Gamma_{h, 0, \rmD}} \sigma^{-1} \avg{\abs{\nabla{w}}}^2 \ud s
+ 
\epsilon \int_{\Gamma_{h, 0, \rmD}} \sigma \abs{\jump{w}}^2 \ud s \right)
\\ \geq &
- \abs{1 + \theta} \, C_\bA \left( (2 + \epsilon^{-1}) \, C \, \alpha^{-1} \sum_{K \in \mT_h} \int_K \abs{\nabla{w}}^2 \ud x
+ 
\epsilon \int_{\Gamma_{h, 0, \rmD}} \sigma \abs{\jump{w}}^2 \ud s \right)
\end{align*}
for any~$\epsilon > 0$. Finally, for~$T_4$, using the Lipschitz condition~\eqref{eq:asm_A1} together with the fact that~$\abs{\theta} \leq 1$, and subsequently applying Lemma~\ref{lem:inv_trace_inequality}, we obtain
\begin{align*}
T_4 
\geq &
- \abs{\theta} C_\bA \int_{\Gamma_{h, 0, \rmD}} \sigma^{-1} \avg{\abs{\nabla{w}}}^2 \ud s
\\ \geq &
- C_\bA \, C \, \alpha^{-1} \sum_{K \in \mT_h} \int_K \abs{\nabla{w}}^2 \ud x .
\end{align*}
Substituting the above bounds for~$T_1$ to~$T_4$ back into~\eqref{eq:N_alt_2} and recalling that~$C_\bA \geq M_\bA > 0$, we deduce that
\begin{align*}
\hspace{20pt} & \hspace{-20pt} N(w_1; w_1 - w_2) - N(w_2; w_1 - w_2) 
\\
\geq & \left( M_\bA - (2 + \epsilon^{-1}) \, \lambda_\theta \, C_\bA \, C \, \alpha^{-1} \right) \sum_{K \in \mT_h} \int_K \abs{\nabla{w_1} - \nabla{w_2}}^2 \ud x 
\\
& + \left( M_\bA - \lambda_\theta \, C_\bA \, \epsilon \right) \int_{\Gamma_{h, 0, \rmD}} \sigma \abs{\jump{w_1 - w_2}}^2 \ud s ,
\end{align*}
where~$\lambda_\theta = 1 + \abs{1 + \theta}$. Upon selecting $\epsilon = M_\bA / (2 \, \lambda_\theta \, C_\bA)$, we arrive at
\begin{align*}
N(w_1; w_1 - w_2) - N(w_2; w_1 - w_2)
\geq \ &
M_\bA \left( 1 - \frac{\alpha_0}{\alpha} \right)
\sum_{K \in \mT_h} \int_K \abs{\nabla{w_1} - \nabla{w_2}}^2 \ud x
\\ &
+
\frac{1}{2} M_\bA
\int_{\Gamma_{h, 0, \rmD}} \sigma \abs{\jump{w_1 - w_2}}^2 \ud s ,
\end{align*}
where~$\alpha_0 = 2 \, C \left( 1 + \lambda_\theta \, C_\bA / M_\bA \right)^2$. Hence, we have proved~\eqref{eq:strong_monotonicity} with $M_N = M_\bA \, \max ( \frac{1}{2}, \, 1 - \alpha_0 / \alpha )$. We conclude by noting that~$M_N > 0$ whenever~$\alpha > \alpha_0$. 
\end{proof}

With the aid of Lemma~\ref{lem:lipschitz_continuity} and Lemma~\ref{lem:strong_monotonicity}, we are now in the position to prove that the DG approximation~\eqref{eq:dgfem} admits a unique solution~$u_{h, p} \in V_{h, p}$. Necessary and sufficient conditions for existence and uniqueness are provided by Theorem~\ref{thm:nonlinear_infsup} in the Appendix. The following result is an immediate consequence.
\begin{theorem}[Existence and uniqueness] \label{thm:wellposedness_dgfem}
Let~$\theta \in [-1, 1]$ and~$\alpha > \alpha_0 = 2 \, C \, ( 1 + \lambda_\theta \, C_\bA / M_\bA )^2$, where $\lambda_\theta = 1 + \abs{1 + \theta}$ and~$C$ is the constant from Lemma~\ref{lem:inv_trace_inequality}. Then, the DG approximation~\eqref{eq:dgfem} has a unique solution~$u_{h, p} \in V_{h, p}$.
\end{theorem}


\section{\emph{A priori} error analysis} \label{section:error_analysis}

We begin by introducing the following $hp$-approximation results.
\begin{lemma} \label{lem:approximation_estimate}
Let $K \in \mT_h$ such that $K = T_K(\hat{K})$, where $\hat{K}$ is either the unit $d$-simplex or the unit $d$-hypercube, and $T_K$ is a $C^{r_K}$-diffeomorphism in compliance with Assumption~\ref{asm:B}(i). For~$s_K \geq 0$, let $v \in H^{s_K}(K)$ and define $t_K = \min(r_K, s_K)$. Then, for $p_K = 1, 2, \dots$, there exists a mapping $\pi_K \colon H^{s_K}(K) \to S_{p_K}(K)$ and a constant~$C$ independent of~$h_K$, $p_K$ and~$v$ such that:
\begin{enumerate}
\item[(i)] for~$0 \leq k \leq t_K$,
\begin{equation*}
\norm{v - \pi_K(v)}_{H^k(K)}
\leq
C \, \frac{h_K^{\mu_K - k}}{p_K^{t_K - k}} \, \norm{v}_{H^{t_K}(K)} \, ;
\end{equation*}
\item[(ii)] for~$0 \leq k + 1/2 < t_K$, and for~$F \in \mF_{h, K}$,
\begin{equation*}
\norm{v - \pi_K(v)}_{H^k(F)}
\leq
C \, \frac{h_K^{\mu_K - k - 1/2}}{p_K^{t_K - k - 1/2}} \, \norm{v}_{H^{t_K}(K)} \, .
\end{equation*}
\end{enumerate}
Here, $\mu_K = \min(p_K + 1, r_K, s_K)$.
\end{lemma}
\begin{proof}
We refer to the proof of Lemma~4.5 in~\cite{BabuskaSuri1987} for the case that~$K$ is an affine image of the unit triangle or unit quadrilateral. The generalization to non-affine triangles and quadrilaterals follows \emph{mutatis mutandis} by proceeding similarly as in the proof of Theorem~1 of~\cite{CiarletRaviart1972} while making use of~\cite[Lemma~4.1]{BabuskaSuri1987}, and subsequently exploiting Assumption~\ref{asm:B}(i). The argument for simplices and hypercubes of dimension $d > 2$ is completely analogous.
\end{proof}

\begin{corollary} \label{crl:approximation_estimate_dgnorm}
For $s > 3/2$, let~$\Pi_{h, p} \colon H^s(\Omega, \mT_h) \to V_{h, p}$ such that $\left. \Pi_{h, p}(\cdot) \right|_K = \pi_K(\cdot)$ for~$K \in \mT_h$, where~$\pi_K$ is the mapping from Lemma~\ref{lem:approximation_estimate}. Moreover, let~$v \in H^s(\Omega, \mT_h)$ with~$\left. v \right|_K \in H^{s_K}(K)$, $s_K \geq s$, $K \in \mT_h$, and select $\alpha > 0$. There exists a constant~$C$ such that 
\begin{equation*}
\dgnorm{v - \Pi_{h, p}(v)}_+
\leq
C \left( \sum_{K \in \mT_h} \frac{h_K^{2 \mu_K - 2}}{p_K^{2 t_K - 3}} \, \norm{v}_{H^{t_K}(K)}^2 \right)^{1/2} ,
\end{equation*}
where~$t_K = \min(r_K, s_K)$ and~$\mu_K = \min(p_K + 1, r_K, s_K)$.
\end{corollary}
\begin{proof}
Consider $K \in \mT_h$ and~$F \in \mF_{h, K}$. From Assumption~\ref{asm:B} it follows that there exists positive constants~$C_1 \equiv C_1(d, \beta_1, \beta_2)$ and~$C_2 \equiv C_2(d, \beta_3)$ such that $C_1^{-1} h_K \leq \mu_F \leq C_1 \, h_K$ and $C_2^{-1} p_K^2 \leq p_F^2 \leq C_2 \, p_K^2$. Hence,
\begin{equation*}
\alpha \, C_3^{-1} \, \frac{p_K^2}{h_K} \leq \left. \sigma \right|_F \leq \alpha \, C_3 \, \frac{p_K^2}{h_K} ,
\end{equation*}
where $C_3 = C_2 / C_1$. Accordingly, by Young's inequality, we have that, for~$\eta = v - \Pi_{h, p}(v)$,
\begin{align*}
\dgnorm{\eta}_+^2 
= \ & 
\sum_{K \in \mT_h} \int_K \abs{\nabla{\eta}}^2 \ud x
+
\int_{\Gamma_{h, 0, \rmD}} \sigma \abs{\jump{\eta}}^2 \ud s
+
\int_{\Gamma_{h, 0, \rmD}} \sigma^{-1} \avg{\abs{\nabla{\eta}}}^2 \ud s 
\\ \leq \ &
\sum_{K \in \mT_h} \left( \norm{\eta}_{H^1(K)}^2
+
\sum_{F \in \mF_{h, K}} \!\! \left( 2 \alpha C_3 \frac{p_K^2}{h_K} \norm{\eta}_{L^2(F)}^2 
+
\alpha^{-1} C_3 \frac{h_K}{p_K^2} \norm{\eta}_{H^1(F)}^2 \right) \right) .
\end{align*}
Here, in view of the approximation estimates from Lemma~\ref{lem:approximation_estimate},
\begin{align*}
\norm{\eta}_{H^1(K)}^2 \leq & \ C \, \frac{h_K^{2 \mu_K - 2}}{p_K^{2 t_K - 2}} \, \norm{v}_{H^{t_K}(K)}^2 , 
\\
\norm{\eta}_{L^2(F)}^2 \leq & \ C \, \frac{h_K^{2 \mu_K - 1}}{p_K^{2 t_K - 1}} \, \norm{v}_{H^{t_K}(K)}^2 , 
\\
\norm{\eta}_{H^1(F)}^2 \leq & \ C \frac{h_K^{2 \mu_K - 3}}{p_K^{2 t_K - 3}} \norm{v}_{H^{t_K}(K)}^2 .
\end{align*}
Hence, 
\begin{equation*}
\dgnorm{\eta}_+^2 
\leq 
C \sum_{K \in \mT_h} \left( \frac{h_K^{2 \mu_K - 2}}{p_K^{2 t_K - 2}} 
+
\alpha \, C_3 \, C_4 \frac{h_K^{2 \mu_K - 2}}{p_K^{2 t_K - 3}}
+
\alpha^{-1} C_3 \, C_4 \frac{h_K^{2 \mu_K - 2}}{p_K^{2 t_K - 1}} \right) \norm{v}_{H^{t_K}(K)}^2 ,
\end{equation*}
where $C_4 = \max_{K \in \mT_h} \left( \card(\mF_{h, K}) \right)$. 
\end{proof}

Using the $hp$-approximation estimate from Corollary~\ref{crl:approximation_estimate_dgnorm}, we prove the following \emph{a priori} error bound.
\begin{theorem} \label{thm:dgnorm_error_estimate}
Let~$u$ denote the solution to~\eqref{eq:model_problem} and suppose that~$u \in H^s(\Omega) \cap C^0(\Omega)$, $s > 3/2$, with $\left. u \right|_K \in H^{s_K}(K)$, $s_K \geq s$, $K \in \mT_h$. Furthermore, let $\theta \in [-1, 1]$ and $\alpha > \alpha_0$, with~$\alpha_0$ as in Lemma~\ref{lem:strong_monotonicity}. Then, denoting by~$u_{h, p} \in V_{h, p}$ the solution to~\eqref{eq:dgfem}, there exists a constant~$C$ such that 
\begin{equation} \label{eq:dgnorm_error_estimate}
\dgnorm{u - u_{h, p}}_+ \leq C \left( \sum_{K \in \mT_h} \frac{h_K^{2 \mu_K - 2}}{p_K^{2 t_K - 3}} \, \norm{u}_{H^{t_K}(K)}^2 \right)^{1/2} ,
\end{equation}
where~$t_K = \min(r_K, s_K)$ and~$\mu_K = \min(p_K + 1, r_K, s_K)$.
\end{theorem}
\begin{proof}
Denote by~$\Pi_{h, p} \colon H^s(\Omega, \mT_h) \to V_{h, p}$ the mapping from Corollary~\ref{crl:approximation_estimate_dgnorm}, and let us write~$u - u_{h, p} = \eta + \xi$, where~$\eta = u - \Pi_{h, p}(u)$ and~$\xi = \Pi_{h, p}(u) - u_{h, p}$. Using Lemma~\ref{lem:strong_monotonicity}, the Galerkin-orthogonality property~\eqref{eq:galerkin_orthogonality} and Lemma~\ref{lem:lipschitz_continuity}, we have that
\begin{align*}
M_N \, \dgnorm{\xi}^2
\leq \ &
N(\Pi_{h, p}(u); \xi) - N(u_{h, p}; \xi)
\\ \leq \ &
N(\Pi_{h, p}(u); \xi) - N(u; \xi) 
\\ \leq \ &
C_N \, \dgnorm{\eta}_+ \, \dgnorm{\xi}_+ .
\end{align*}
Since~$\xi \in V_{h, p}$, we note from~\eqref{eq:norm_equivalence} that there exists a constant~$C$ such that~$\dgnorm{\xi}_+^2 \leq C \, \dgnorm{\xi}^2$. Hence,
\begin{equation*}
\dgnorm{\xi}_+ \leq C \, \frac{C_N}{M_N} \, \dgnorm{\eta}_+ ,
\end{equation*}
and therefore, by the triangle inequality,
\begin{equation*}
\dgnorm{u - u_{h, p}}_+ 
\leq 
\dgnorm{\eta}_+ + \dgnorm{\xi}_+
\leq \left( 1 + C \, \frac{C_N}{M_N} \right) \dgnorm{\eta}_+ .
\end{equation*}
The estimate~\eqref{eq:dgnorm_error_estimate} then follows by applying Corollary~\ref{crl:approximation_estimate_dgnorm}.
\end{proof}

We remark that the error estimate obtained in Theorem~\ref{thm:dgnorm_error_estimate} displays the same quasi-optimality as the error estimates obtained for interior penalty DG approximations of linear elliptic problems; cf., for example, \cite[Theorem~4.5]{HoustonEtAl2002}. That is, provided that $r_K \geq s_K \geq p_K + 1$ for all~$K \in \mT_h$, the estimate~\eqref{eq:dgnorm_error_estimate} is optimal in~$h$ and slightly suboptimal in~$p$, by half an order in~$p$. Here, the condition that~$r_K \geq s_K$ for all~$K \in \mT_h$ reflects the dependence of the estimates on the regularity of the mappings~$\{ T_K \}_{K \in \mT_h}$, and stresses the importance of proper mesh design, especially when curved elements are used; cf.~\cite{CiarletRaviart1972}.

Next, let~$\psi \in L^2(\Omega)$ and consider the linear functional $J_\psi(w) = (\psi, w)_\Omega$, where $w \in V(h, p)$ and $(\cdot, \cdot)_\Omega$ denotes the $L^2(\Omega)$~inner product. We shall now be concerned with obtaining a bound for the error $J_\psi(u) - J_\psi(u_{h, p})$. The analysis is based on a duality argument and relies on Fr\'{e}chet differentiability of the map~$\bv \mapsto \bA(\bv) \bv \colon \mathbb{R}^d \to \mathbb{R}^d$ with respect to~$\bv$.
Accordingly, if the limit exits, let us denote by
\begin{equation} \label{eq:da}
\ba'(\bq; \bw) := \lim_{t \to 0} \frac{\bA(\bq + t \bw) (\bq + t \bw) - \bA(\bq) \bq}{t} , \qquad \bq, \bw \in \mathbb{R}^d ,
\end{equation}
the derivative of the map $\bv \mapsto \bA(\bv) \bv \colon \mathbb{R}^d \to \mathbb{R}^d$ at~$\bq$ in the direction~$\bw$. Thanks to Assumption~\ref{asm:A} we are able to make the following claim.
\begin{lemma} \label{lem:frechet_differentiability}
Let $\bA$ satisfy the Lipschitz condition~\eqref{eq:asm_A1} of Assumption~\ref{asm:A}. Then, the map~$\bv \mapsto \bA(\bv) \bv \colon \mathbb{R}^d \to \mathbb{R}^d$ is Fr\'{e}chet differentiable almost everywhere. That is, for almost every~$\bq \in \mathbb{R}^d$, we have that:
\begin{enumerate}
\item[(i)] the limit~\eqref{eq:da} exists for all~$\bw \in \mathbb{R}^d$;
\item[(ii)] the mapping~$\bw \mapsto \ba'(\bq; \bw) \colon \mathbb{R}^d \to \mathbb{R}^d$ is linear and continuous;
\item[(iii)] $\ba'(\bq; \bw) = \bA(\bq + \bw) (\bq + \bw) - \bA(\bq) \bq + o(\abs{\bw})$ as $\bw \to \mathbf{0}$ in $\mathbb{R}^d$.
\end{enumerate}
\end{lemma}
\begin{proof}
The lemma is an immediate consequence of Rademacher's Theorem; see, for example, \cite[Section~3.1.2]{EvansGariepy1992}.
\end{proof}

For simplicity of presentation, and without loss of generality, we henceforth assume that the map~$\bv \mapsto \bA(\bv) \bv \colon \mathbb{R}^d \to \mathbb{R}^d$ is \emph{everywhere} Fr\'{e}chet differentiable in~$\mathbb{R}^d$, and we refer to Remark~\ref{rem:regularization} below for further discussion. Then, for $\bq \in \mathbb{R}^d$, let $\bA^\ast(\bq) \in \mathbb{R}^{d, d}$ such that~$\bA^\ast(\bq) \bv \cdot \bw = \ba'(\bq; \bw) \cdot \bv$ for all~$\bv, \bw \in \mathbb{R}^d$. Given~$\psi \in L^2(\Omega)$, we introduce the dual problem: find~$z \colon \Omega \to \mathbb{R}$ such that
\begin{subequations} \label{eq:dual_problem}
\begin{alignat}{2}
- \nabla{} \cdot \left( \bA^\ast(\nabla{u}) \nabla{z} \right) =  \ & \psi & \qquad & \text{in} \ \Omega , \label{eq:dual_problem_pde}
\\
z = \ & 0 & \qquad & \text{on} \ \Gamma_\rmD , \label{eq:dual_problem_bcD}
\\
\bA^\ast(\nabla{u}) \nabla{z} \cdot \bn = \ & 0 & \qquad & \text{on} \ \Gamma_\rmN . \label{eq:dual_problem_bcN}
\end{alignat}
\end{subequations}
Using Assumption~\ref{asm:A}, it is easy verify that~$\abs{\bA^\ast(\bq) \bv} \leq C_\bA \abs{\bv}$ and $\bA^\ast(\bq) \bv \cdot \bv \geq M_\bA \abs{\bv}^2$ for all~$\bq, \bv \in \mathbb{R}^d$, where~$C_\bA$ and~$M_\bA$ are the constants from~\eqref{eq:asm_A1} and~\eqref{eq:asm_A2}. 
%
%
Hence, by the Lax-Milgram theorem we deduce that~\eqref{eq:dual_problem} has a unique weak solution~$z \in H^1(\Omega)$. In what follows, we shall assume slightly stronger regularity by supposing that there exists a strong solution $z \in H^2(\Omega)$ satisfying
\begin{equation} \label{eq:dual_regularity}
\norm{z}_{H^2(\Omega)} \leq C \norm{\psi}_{L^2(\Omega)} .
\end{equation}
From~\cite[Theorem~8.12]{GilbargTrudinger1983}, we note that this is satisfied if $\partial \Omega$ is of class~$C^2$ with~$\Gamma_\rmN = \emptyset$, and if~$\bA^\ast(\nabla{u}) \in [C^{0, 1}(\closure{\Omega})]^{d, d}$.

With the aid of the dual problem~\eqref{eq:dual_problem} we are able to derive the following \emph{a priori} bound for the error~$J_\psi(u) - J_\psi(u_{h, p})$.
\begin{theorem} \label{thm:functional_error_estimate}
Consider the same premises as in Theorem~\ref{thm:dgnorm_error_estimate}. Furthermore, assume that the map~$\bv \mapsto \bA(\bv) \bv \colon \mathbb{R}^d \to \mathbb{R}^d$ is everywhere Fr\'{e}chet differentiable in $\mathbb{R}^d$, and given $\psi \in L^2(\Omega)$, suppose that the dual problem~\eqref{eq:dual_problem} has a strong solution~$z \in H^2(\Omega)$ with $\left. z \right|_K \in H^{\ell_K}(K)$, $\ell_K \geq 2$, $K \in \mT_h$. Then, there exists a constant~$C$ such that
\begin{align}
J_\psi(u) - J_\psi(u_{h, p}) 
\leq &
\ C \left( \sum_{K \in \mT_h} \frac{h_K^{2 \mu_K - 2}}{p_K^{2 t_K - 3}} \norm{u}_{H^{t_K}(K)}^2 \right)^{\!\!1/2} 
\nonumber \\
& \times \! \left(
\left( \sum_{K \in \mT_h} \frac{h_K^{2 \lambda_K - 2}}{p_K^{2 m_K - 3}} \norm{z}_{H^{m_K}(K)}^2 \right)^{\!\!1/2} 
+
\frac{1 + \theta}{\sqrt{\alpha}} \, \norm{z}_{H^2(\Omega)} \! \right) + R , \label{eq:functional_error_estimate}
\end{align}
where $t_K = \min(r_K, s_K)$, $m_K = \min(r_K, \ell_K)$, $\mu_K = \min(p_K + 1, r_K, s_K)$, $\lambda_K = \min(p_K + 1, r_K, \ell_K)$, and where $R = o( \dgnorm{u - u_{h, p}}_+ ) \, \norm{z}_{H^2(\Omega)}$. Moreover, if the map~$\bv \mapsto \bA(\bv) \bv \colon \mathbb{R}^d \to \mathbb{R}^d$ is twice continuously differentiable everywhere in~$\mathbb{R}^d$, then there exists a constant~$C$ such that
\begin{equation}
R \leq C \max_{K \in \mT_h} \left( \frac{p_K^{3/2}}{h_K^{d/2}} \right) \left( \sum_{K \in \mT_h} \frac{h_K^{2 \mu_K - 2}}{p_K^{2 t_K - 3}} \norm{u}_{H^{t_K}(K)}^2 \! \right) \norm{z}_{H^2(\Omega)} . \label{eq:functional_error_estimate_Rbound}
\end{equation}
\end{theorem}

Before we embark on the proof of Theorem~\ref{thm:functional_error_estimate}, we first introduce an auxiliary result. By our assumption that the map~$\bv \mapsto \bA(\bv) \bv \colon \mathbb{R}^d \to \mathbb{R}^d$ is everywhere Fr\'{e}chet differentiable in~$\mathbb{R}^d$, we have that the map $y \mapsto N(y; v) \colon V(h, p) \to \mathbb{R}$ is everywhere Fr\'{e}chet differentiable in~$V(h, p)$. Accordingly, for any~$v \in V(h, p)$, let~$N'(q; w, v)$ denote the derivative of the map $y \mapsto N(y; v) \colon V(h, p) \to \mathbb{R}$ at~$q$ in the direction~$w$, given by
\begin{equation*}
N'(q; w, v) = \lim_{t \to 0} \frac{N(q + t w; v) - N(q; v)}{t} , \qquad q, w, v \in V(h, p) .
\end{equation*}
%
%
%
We introduce the following auxiliary result.
\begin{lemma} \label{lem:dual_consistency}
Let~$u \in H^s(\Omega) \cap C^0(\Omega)$, $s > 3/2$, denote the solution of~\eqref{eq:model_problem}, and suppose that the dual problem~\eqref{eq:dual_problem} has a strong solution~$z \in H^2(\Omega)$. Then, 
\begin{equation*}
J_\psi(w) 
= 
N'(u; w, z) 
- 
(1 + \theta) \int_{\Gamma_{h, 0, \rmD}} \ba'(\nabla{u}; \jump{w}) \cdot \nabla{z} \ud s \qquad \forall w \in V(h, p) .
\end{equation*}
\end{lemma}
\begin{proof}
Since~$z \in H^2(\Omega)$, we have that~$\left. \jump{z} \right|_F = \mathbf{0}$ for all~$F \in \mF_{h, 0}$. Accordingly, evaluating~$N'(u; w, z)$ for any~$w \in V(h, p)$, we find that
\begin{equation}
N'(u; w, z) 
=
\sum_{K \in \mT_h} \int_K \ba'(\nabla{u}; \nabla{w}) \cdot \nabla{z} \ud x
+
\theta \int_{\Gamma_{h, 0, \rmD}} \! \avg{\ba'(\nabla{u}; \jump{w}) \cdot \nabla{z}} \ud s . \label{eq:dN_uwz}
\end{equation}
Using the dual problem~\eqref{eq:dual_problem} and applying integration-by-parts, we also find that, for all~$w \in V(h, p)$,
\begin{equation} \label{eq:Jpsi_w}
\begin{aligned}
J_\psi(w)
= \ &
- \sum_{K \in \mT_h} \int_K w \, \nabla \cdot \left( \bA^\ast(\nabla{u}) \nabla{z} \right) \ud x
\\ = \ &
\sum_{K \in \mT_h} \left( \int_K \bA^\ast(\nabla{u}) \nabla{z} \cdot \nabla{w} \ud x - \int_{\partial K} \bA^\ast(\nabla{u}) \nabla{z} \cdot \bn_K w \ud s \right)
\\ = \ &
\sum_{K \in \mT_h} \int_K \bA^\ast(\nabla{u}) \nabla{z} \cdot \nabla{w} \ud x 
-
\int_{\Gamma_{h, 0}} \jump{\bA^\ast(\nabla{u}) \nabla{z}} \avg{w} \ud s
\\ & -
\int_{\Gamma_{h, 0, \rmD}} \avg{\bA^\ast(\nabla{u}) \nabla{z}} \cdot \jump{w} \ud s .
\end{aligned}
\end{equation}
By \cite[Lemma~1.24]{PietroErn2011}, it follows that $\left. \jump{\bA^\ast(\nabla{u}) \nabla{z}} \right|_F = 0$ weakly for all~$F \in \mF_{h, 0}$. Thence, comparing~\eqref{eq:dN_uwz} and~\eqref{eq:Jpsi_w} while noting that $\bA^\ast(\nabla{u}) \nabla{z} \cdot \bw = \ba'(\nabla{u}; \bw) \cdot \nabla{z}$ for all~$\bw \in \mathbb{R}^d$, we obtain the stated result.
\end{proof}

With the aid of Lemma~\ref{lem:dual_consistency}, we now present a proof of Theorem~\ref{thm:functional_error_estimate}.
\begin{proof}[Proof of Theorem~\ref{thm:functional_error_estimate}]
Denote by~$\Pi_{h, p} \colon H^s(\Omega, \mT_h) \to V_{h, p}$, $s > 3/2$, the mapping from Corollary~\ref{crl:approximation_estimate_dgnorm}, and let us write $e = u - u_{h, p}$. Lemma~\ref{lem:dual_consistency} implies that
\begin{align}
J_\psi(u) - J_\psi(u_{h, p}) 
= \ & 
N'(u; e, z) 
- 
(1 + \theta) \int_{\Gamma_{h, 0, \rmD}} \ba'(\nabla{u}; \jump{e}) \cdot \nabla{z} \ud s \nonumber
\\ = \ &
N'(u; e, z - \Pi_{h, p}(z)) 
- 
(1 + \theta) \int_{\Gamma_{h, 0, \rmD}} \ba'(\nabla{u}; \jump{e}) \cdot \nabla{z} \ud s \nonumber
\\ &
+
N'(u; e, \Pi_{h, p}(z)) . \label{eq:dual_error_representation}
\end{align}
Considering the first term in~\eqref{eq:dual_error_representation}, we deduce by Lemma~\ref{lem:lipschitz_continuity} that
\begin{align*}
N'(u; e, z - \Pi_{h, p}(z)) 
= \ &
\lim_{t \to 0} \frac{ N(u + t e; z - \Pi_{h, p}(z)) - N(u; z - \Pi_{h, p}(z))}{t}
\\ \leq \ &
\sup_{t > 0} \frac{ N(u + t e; z - \Pi_{h, p}(z)) - N(u; z - \Pi_{h, p}(z)) }{ t }
\\ \leq \ &
C_N \, \dgnorm{e}_+ \, \dgnorm{z - \Pi_{h, p}(z)}_+ ,
\end{align*}
where~$C_N$ is the constant from Lemma~\ref{lem:lipschitz_continuity}. Using the error estimate of Theorem~\ref{thm:dgnorm_error_estimate} and the approximation estimate of Corollary~\ref{crl:approximation_estimate_dgnorm}, we then obtain:
\begin{align*}
& N'(u; e, z - \Pi_{h, p}(z))
\\ & \quad \leq
C \left( \sum_{K \in \mT_h} \frac{h_K^{2 \mu_K - 2}}{p_K^{2 t_K - 3}} \norm{u}_{H^{t_K}(K)}^2 \right)^{\!1/2} 
\left( \sum_{K \in \mT_h} \frac{h_K^{2 \lambda_K - 2}}{p_K^{2 m_K - 3}} \norm{z}_{H^{m_K}(K)}^2 \right)^{\!1/2} . 
\end{align*}
Next, applying the Cauchy-Schwarz inequality to the second term in~\eqref{eq:dual_error_representation}, we have that
\begin{align*}
& (1 + \theta) \int_{\Gamma_{h, 0, \rmD}} \ba'(\nabla{u}; \jump{u - u_{h, p}}) \cdot \nabla{z} \ud s
\\ & \qquad \leq
(1 + \theta) 
\left( \int_{\Gamma_{h, 0, \rmD}} \sigma \abs{\ba'(\nabla{u}; \jump{u - u_{h, p}})}^2 \ud s \right)^{1/2}
\left( \int_{\Gamma_{h, 0, \rmD}} \sigma^{-1} \abs{\nabla{z}}^2 \right)^{1/2} .
\end{align*}
Using that $\abs{\ba'(\bq; \bw)} \leq C_\bA \abs{\bw}$ for all~$\bq, \bw \in \mathbb{R}^d$ and subsequently applying Theorem~\ref{thm:dgnorm_error_estimate}, we find:
\begin{equation*}
\int_{\Gamma_{h, 0, \rmD}} \sigma \abs{\ba'(\nabla{u}; \jump{e})}^2 \ud s
\, \leq \,
C_\bA \dgnorm{e}^2
\, \leq \,
C \left( \sum_{K \in \mT_h} \frac{h_K^{2 \mu_K - 2}}{p_K^{2 t_K - 3}} \norm{u}_{H^{t_K}(K)}^2 \right) .
\end{equation*}
Moreover, argueing similarly as in the proof of Lemma~\ref{lem:inv_trace_inequality} and subsequently applying the trace inequality from Lemma~\ref{lem:trace_inequality}, we deduce that
\begin{align}
\int_{\Gamma_{h, 0, \rmD}} \sigma^{-1} \abs{\nabla{z}}^2 \ud s
\leq \ &
C \alpha^{-1} \sum_{K \in \mT_h} \frac{h_K}{p_K^2} \int_{\partial K} \abs{\nabla{z}}^2 \ud s \nonumber
\\ \leq \ &
C \alpha^{-1} \sum_{K \in \mT_h} \frac{h_K}{p_K^2} \left( h_K^{-1} \norm{z}_{H^1(K)}^2 + \norm{z}_{H^1(K)} \, \norm{z}_{H^2(K)} \right) \nonumber
\\ \leq \ &
C \alpha^{-1} \, \norm{z}_{H^2(\Omega)}^2 . \label{eq:dual_trace_inequality}
\end{align}
Hence, we obtain:
\begin{equation*}
(1 + \theta) \int_{\Gamma_{h, 0, \rmD}} \! \ba'(\nabla{u}; \jump{e}) \cdot \nabla{z} \ud s 
\leq
C \, \frac{1 + \theta}{\sqrt{\alpha}} \left( \sum_{K \in \mT_h} \frac{h_K^{2 \mu_K - 2}}{p_K^{2 t_K - 3}} \norm{u}_{H^{t_K}(K)}^2 \! \right)^{\!1/2} \norm{z}_{H^2(\Omega)} .
\end{equation*}
Substituting the above bounds back into~\eqref{eq:dual_error_representation}, we arrive at the stated estimate~\eqref{eq:functional_error_estimate} with~$R = N'(u; e, \Pi_{h, p}(z))$. 

We claim that $R = o( \dgnorm{e}_+ ) \, \norm{z}_{H^2(\Omega)}$. Fr\'{e}chet differentiability of the map $y \mapsto N(y; v) \colon V(h, p) \to \mathbb{R}$ everywhere in~$V(h, p)$ implies that
\begin{equation*}
N'(q; w, v) = N(q + w; v) - N(q; v) + o( \dgnorm{w}_+ ) \, \dgnorm{v}_+ \qquad \text{as $\dgnorm{w}_+ \to 0$} , 
\end{equation*}
for all~$q, w, v \in V(h, p)$. Hence, by the Galerkin-orthogonality property of Lemma~\ref{lem:galerkin_orthogonality}, we obtain that
\begin{align*}
R = N'(u; e, \Pi_{h, p}(z)) 
= \ &
%
N(u; \Pi_{h, p}(z)) - N(u_{h, p}; \Pi_{h, p}(z)) 
+ 
o( \dgnorm{e}_+ ) \ \dgnorm{\Pi_{h, p}(z)}_+ 
\\ = \ &
o( \dgnorm{e}_+ ) \ \dgnorm{\Pi_{h, p}(z)}_+ 
\end{align*}
as $\dgnorm{e}_+ \to 0$. Here, in view of~\eqref{eq:dual_trace_inequality}, we have that~$\dgnorm{z}_+ \leq C \norm{z}_{H^2(\Omega)}$, so that, by the triangle inequality and Corollary~\ref{crl:approximation_estimate_dgnorm},
\begin{equation}
\dgnorm{\Pi_{h, p}(z)}_+ \leq C \norm{z}_{H^2(\Omega)} . \label{eq:dgnormPiz}
\end{equation}
Therefore, we find that $R = o( \dgnorm{e}_+ ) \, \norm{z}_{H^2(\Omega)}$, as claimed.

It remains to prove the estimate~\eqref{eq:functional_error_estimate_Rbound} subject to the condition that the map $\bv \mapsto \bA(\bv) \bv \colon \mathbb{R}^d \to \mathbb{R}^d$ is twice continuously differentiable everywhere in~$\mathbb{R}^d$. Accordingly, let
\begin{equation*}
\ba''(\bq; \bw_1, \bw_2) := \lim_{t \to 0} \frac{\ba'(\bq + t \bw_2; \bw_1) - \ba'(\bq; \bw_1)}{t} , \qquad \bq, \bw_1, \bw_2 \in \mathbb{R}^d ,
\end{equation*}
denote the second-order derivative of the map~$\bv \mapsto \bA(\bv) \bv \colon \mathbb{R}^d \to \mathbb{R}^d$ at~$\bq \in \mathbb{R}^d$ in the direction~$(\bw_1, \bw_2) \in \mathbb{R}^d \times \mathbb{R}^d$, and let there be a constant~$C'_\bA$ such that $\abs{\ba''(\bq; \bw_1, \bw_2)} \leq C'_\bA \, \abs{\bw_1} \, \abs{\bw_2}$ for all $\bq, \bw_1, \bw_2 \in \mathbb{R}^d$. By Taylor's Theorem, we have that
\begin{equation}
\bA(\bv_1) \bv_1 - \bA(\bv_2) \bv_2 = \ba'(\bv_2; \bv_1 - \bv_2) + \br(\bv_2, \bv_1 - \bv_2) , \qquad \forall \bv_1, \bv_2 \in \mathbb{R}^d , \label{eq:taylor_expansion}
\end{equation}
with the integral remainder
\begin{equation*}
\br(\bv_2, \bv_1 - \bv_2) 
= 
\int_0^1 \ba''( \bv_2 + t (\bv_1 - \bv_2); \bv_1 - \bv_2, \bv_1 - \bv_2) (1 - t) \ud t ,
\end{equation*}
satisfying $\abs{\br(\bv_2, \bv_1 - \bv_2)} \leq C'_\bA \abs{\bv_1 - \bv_2}^2$. Now, recall that $R = N'(u; e, \Pi_{h, p}(z))$. Using the Galerkin-orthogonality property of Lemma~\ref{lem:galerkin_orthogonality} and the Taylor expansion~\eqref{eq:taylor_expansion}, we deduce that
\begin{align*}
R 
= \ &
N'(u; e, \Pi_{h, p}(z))
- 
N(u; \Pi_{h, p}(z)) 
+ 
N(u_{h, p}; \Pi_{h, p}(z))
\\ = \ &
- \sum_{K \in \mT_h} \int_K \br(\nabla{u}; \nabla{e}) \cdot \nabla{(\Pi_{h, p}(z))} \ud x 
\\ &
+ \int_{\Gamma_{h, 0, \rmD}} \avg{ \br(\nabla{u}, \widehat{\nabla}_\sigma \, e) \cdot (\theta \sigma^{-1} \nabla{(\Pi_{h, p}(z))} + \jump{\Pi_{h, p}(z)}) } \ud s
\\ &
- \theta \int_{\Gamma_{h, 0, \rmD}} \sigma^{-1} \avg{ \br(\nabla{u}, \nabla{e}) \cdot \nabla({\Pi_{h, p}(z)}) } \ud s 
\\ \leq \ &
C'_\bA \sum_K \int_K \abs{\nabla{e}}^2 \ \abs{\nabla{(\Pi_{h, p}(z))}} \ud x
\\ & 
+ C'_\bA \int_{\Gamma_{h, 0, \rmD}} \avg{ \abs{\widehat{\nabla}_\sigma \, e}^2 \, \left( \abs{\theta} \, \sigma^{-1} \abs{\nabla(\Pi_{h, p}(z))} + \abs{\jump{\Pi_{h, p}(z)}} \right) } \ud s 
\\ &
+ C'_\bA \, \abs{\theta} \int_{\Gamma_{h, 0, \rmD}} \sigma^{-1} \avg{ \abs{\nabla{e}}^2 \ \abs{\nabla{(\Pi_{h, p}(z))}} } \ud s .
\end{align*}
By Young's inequality and the fact that~$\abs{\avg{\bq_1 \cdot \bq_2}} \leq \avg{\abs{\bq_1} \ \abs{\bq_2}} \leq 2 \avg{\abs{\bq_1}} \, \avg{\abs{\bq_2}}$ for all~$\bq_1, \bq_2 \in [H^1(\Omega, \mT_h)]^d$, we then obtain:
\begin{align}
R 
\leq \ &
%
C'_\bA \sum_K \int_K \abs{\nabla{e}}^2 \ \abs{\nabla{(\Pi_{h, p}(z))}} \ud x \nonumber
\\ & 
+ 12 \abs{\theta} \, C'_\bA \int_{\Gamma_{h, 0, \rmD}} \left( \sigma^{-1} \avg{\abs{\nabla{e}}}^2 + \sigma \abs{\jump{e}}^2 \right) \, \avg{\abs{\nabla{(\Pi_{h, p}(z))}}} \ud s \nonumber
\\ &
+ 4 \, C'_\bA \int_{\Gamma_{h, 0, \rmD}} \left( \avg{\abs{\nabla{e}}}^2 + \sigma^2 \abs{\jump{e}}^2 \right) \, \abs{\jump{\Pi_{h, p}(z)}} \ud s \nonumber
\\ \leq \ &
(4 + 12 \abs{\theta}) \ C'_\bA \ \dgnorm{e}_+^2 \ \dgnorm{\Pi_{h, p}(z)}_\star ,  \label{eq:Rbound_star}
\end{align}
where
\begin{align}
\dgnorm{\Pi_{h, p}(z)}_\star
= \ &
\max_{K \in \mT_h} \norm{\Pi_{h, p}(z)}_{W^1_\infty(K)}
+
\max_{F \in \mF_{h, 0, \rmD}} \bignorm{\avg{\abs{\nabla{(\Pi_{h, p}(z))}}}}_{L^\infty(F)} 
\nonumber \\ & 
+ 
\max_{F \in \mF_{h, 0, \rmD}} \sigma \, \bignorm{ \abs{\jump{\Pi_{h, p}(z)}} }_{L^\infty(F)} . \label{eq:norm_Piz_star}
\end{align}
An upper bound for $\dgnorm{e}_+$ is provided by Theorem~\ref{thm:dgnorm_error_estimate}. To prove~\eqref{eq:functional_error_estimate_Rbound}, it thus remains to show that $\dgnorm{\Pi_{h, p}(z)}_\star \leq C \max_{K \in \mT_h} \Big( p_K^{3/2} \, h_K^{-d/2} \Big) \norm{z}_{H^2(\Omega)}$. 
To this end, let us note that, in view of Lemma~\ref{lem:approximation_estimate} and the triangle inequality, there exists a constant~$C$ such that $\norm{\Pi_{h, p}(z)}_{H^2(K)} \leq C \norm{z}_{H^2(K)}$. Thence, exploiting the inverse estimate~\eqref{eq:inv_estimate_2a}, we have that
\begin{align*}
\max_{K \in \mT_h} \norm{\Pi_{h, p}(z)}_{W^1_\infty(K)} 
\leq & \
C \max_{K \in \mT_h} \left( \frac{p_K}{h_K^{d / 2}} \, \norm{\Pi_{h, p}(z)}_{H^1(K)} \right)
\\ \leq & \
C \max_{K \in \mT_h} \left( \frac{p_K}{h_K^{d / 2}} \right) \norm{z}_{H^2(\Omega)} . 
\end{align*}
For the second term in~\eqref{eq:norm_Piz_star}, we apply the inverse estimate~\eqref{eq:inv_estimate_2b} to obtain
\begin{align*}
& \max_{F \in \mF_{h, 0, \rmD}} \bignorm{\avg{\abs{\nabla{(\Pi_{h, p}(z))}}}}_{L^\infty(F)} 
\\ & \qquad \leq
\max_{K \in \mT_h} \left( \max_{F \in \mF_{h, K}} \bignorm{ \left. (\Pi_{h, p}(z)) \right|_K }_{W^1_\infty(F)} \right)
\\ & \qquad \leq
C \max_{K \in \mT_h} \left( \max_{F \in \mF_{h, K}} \frac{p_K}{(\meas[d-1]{F})^{1/2}} \, \bignorm{ \left. (\Pi_{h, p}(z)) \right|_K }_{H^1(F)} \right) .
\end{align*}
On account of Assumption~\ref{asm:B}, there exists a constant~$C \equiv C(d, \beta_1, \beta_2)$ such that~$\meas[d-1]{F} \geq C \, h_K^{d-1}$ for all~$F \in \mF_{h, K}$, $K \in \mT_h$. Applying the trace inequality~\eqref{eq:trace_inequality}, we then find that
\begin{align*}
& \max_{F \in \mF_{h, 0, \rmD}} \bignorm{\avg{\abs{\nabla{(\Pi_{h, p}(z))}}}}_{L^\infty(F)} 
\\ & \qquad \leq
C \max_{K \in \mT_h} \frac{p_K}{h_K^{d/2}} \left( \norm{\Pi_{h, p}(z)}_{H^1(K)}^2 + h_K \, \norm{\Pi_{h, p}(z)}_{H^1(K)} \, \norm{\Pi_{h, p}(z)}_{H^2(K)} \right)^{1/2}
\\ & \qquad \leq
C \max_{K \in \mT_h} \left( \frac{p_K}{h_K^{d/2}} \right) \norm{z}_{H^2(\Omega)} .
\end{align*}
Finally, considering the third term in~\eqref{eq:norm_Piz_star}, we deduce that, by Assumption~\ref{asm:B} and the inverse estimate~\eqref{eq:inv_estimate_2b},
\begin{align*}
& \max_{F \in \mF_{h, 0, \rmD}} \sigma \, \bignorm{ \abs{\jump{\Pi_{h, p}(z)}} }_{L^\infty(F)}
\\ & \qquad = \
\max_{K \in \mT_h} \left( \max_{F \in \mF_{h, K} \cap \mF_{h, 0, \rmD}} \sigma \bignorm{ \abs{\jump{\Pi_{h, p}(z)}} }_{L^\infty(F)} \right)
\\ & \qquad \leq \
C \max_{K \in \mT_h} \left( \frac{p_K^3}{h_K^{(d+1)/2}} \max_{F \in \mF_{h, K} \cap \mF_{h, 0, \rmD}} \bignorm{ \abs{\jump{\Pi_{h, p}(z)}} }_{L^2(F)} \right) .
\end{align*}
By the fact that $z \in H^1(\Omega)$ with~$z = 0$ on~$\Gamma_\rmD$, we have that~$\bignorm{\abs{\jump{\Pi_{h, p}(z)}}}_{L^2(F)} = \bignorm{\abs{\jump{z - \Pi_{h, p}(z)}}}_{L^2(F)}$ for all~$F \in \mF_{h, 0, \rmD}$. Applying Lemma~\ref{lem:approximation_estimate}, we then obtain:
\begin{align*}
\max_{F \in \mF_{h, 0, \rmD}} \sigma \, \bignorm{ \abs{\jump{\Pi_{h, p}(z)}} }_{L^\infty(F)}
\leq \ &
C \max_{K \in \mT_h} \left( \frac{p_K^3}{h_K^{(d+1)/2}} \max_{F \in \mF_{h, K}} \bignorm{z - (\Pi_{h, p}(z)) |_K }_{L^2(F)} \right)
\\ \leq \ &
C \max_{K \in \mT_h} \left( \frac{p_K^{3/2}}{h_K^{(d-2)/2}} \right) \norm{z}_{H^2(\Omega)} .
\end{align*}
Substituting the above inequalities back into~\eqref{eq:norm_Piz_star}, we thus find that $\dgnorm{\Pi_{h, p}(z)}_\star \leq C \max_{K \in \mT_h}\left( p_K^{3/2} \, h_K^{-d/2} \right) \norm{z}_{H^2(\Omega)}$, which, by~\eqref{eq:Rbound_star}, brings us to the stated result~\eqref{eq:functional_error_estimate_Rbound}.
\end{proof}

As a corollary to Theorem~\ref{thm:functional_error_estimate}, we obtain the following estimate for the error in the $L^2(\Omega)$-norm.
\begin{corollary} \label{crl:l2norm_error_estimate}
Consider the same premises as in Theorem~\ref{thm:functional_error_estimate} and assume that the dual regularity estimate~\eqref{eq:dual_regularity} holds. Then, there exists a constant~$C$ such that
\begin{equation}
\norm{u - u_{h, p}}_{L^2(\Omega)} 
\leq C \left( \frac{h}{p^{1/2}} + \frac{1+\theta}{\sqrt{\alpha}} \right) \, \left( \sum_{K \in \mT_h} \frac{h_K^{2 \mu_K - 2}}{p_K^{2 t_K - 3}} \norm{u}_{H^{t_K}(K)}^2 \! \right)^{\!1/2} 
+ R , \label{eq:l2norm_error_estimate}
\end{equation}
where $t_K = \min(r_K, s_K)$, $\mu_K = \min(p_K + 1, r_K, s_K)$ and $R = o( \dgnorm{u - u_{h, p}}_+ )$. Moreover, if the map~$\bv \mapsto \bA(\bv) \colon \mathbb{R}^d \to \mathbb{R}^d$ is twice continuously differentiable everywhere in~$\mathbb{R}^d$, then there exists a constant~$C$ such that
\begin{equation}
R \leq C \max_{K \in \mT_h} \left( \frac{p_K^{3/2}}{h_K^{d/2}} \right) \, \left( \sum_{K \in \mT_h} \frac{h_K^{2 \mu_K - 2}}{p_K^{2 t_K - 3}} \norm{u}_{H^{t_K}(K)}^2 \! \right) . \label{eq:l2norm_error_estimate_Rbound}
\end{equation}
\end{corollary}
\begin{proof}
The result follows immediately from Theorem~\ref{thm:functional_error_estimate} by selecting~$\psi = u - u_{h, p}$ and subsequently applying the regularity estimate~\eqref{eq:dual_regularity}.
\end{proof}

Let us briefly discuss the error estimates presented in Theorem~\ref{thm:functional_error_estimate} and Corollary~\ref{crl:l2norm_error_estimate}. For $h / p$ sufficiently small, we observe that
\begin{equation*}
J_\psi(u) - J_\psi(u_{h, p})
\leq
C \left( \frac{h^{\mu + \lambda - 2}}{p^{t + m - 3}} + \frac{1 + \theta}{\sqrt{\alpha}} \frac{h^{\mu-1}}{p^{t - 3/2}} \right) \norm{u}_{H^t(\Omega)} \, \norm{z}_{H^m(\Omega)}
\end{equation*}
and
\begin{equation*}
\norm{u - u_{h, p}}_{L^2(\Omega)}
\leq
C \left( \frac{h^{\mu}}{p^{t - 1}} + \frac{1 + \theta}{\sqrt{\alpha}} \frac{h^{\mu-1}}{p^{t - 3/2}} \right) \norm{u}_{H^t(\Omega)} ,
\end{equation*}
where~$t = \min_{K \in \mT_h}(t_K)$, $m = \min_{K \in \mT_h}(m_K)$, $\mu = \min_{K \in \mT_h}(\mu_K)$ and $\lambda = \min_{K \in \mT_h}(\lambda_K)$. Accordingly, when~$\theta = -1$, we find that both estimates are optimal in~$h$ and slightly suboptimal in~$p$, by one order in~$p$. On the other hand, when~$\theta \neq -1$, we find that the estimates are suboptimal in both~$h$ and~$p$, by a factor of respectively~$h^{\lambda - 1} / p^{m - 3/2}$ and~$h / p^{1/2}$. This suboptimality can be attributed to a lack of dual consistency; see Lemma~\ref{lem:dual_consistency}. We note that, for $h / p$ sufficiently small, the above estimates are identical to those obtained for interior penalty DG approximations of linear elliptic problems; cf.~\cite[Theorem~4.4]{HarrimanEtAl2003}.

%
\begin{remark} \label{rem:regularization}
For the proof of Theorem~\ref{thm:functional_error_estimate} and Corollary~\ref{crl:l2norm_error_estimate} we assumed that the map~$\bv \mapsto \bA(\bv) \colon \mathbb{R}^d \to \mathbb{R}^d$ is Fr\'{e}chet differentiable everywhere in~$\mathbb{R}^d$. This was done in order to ensure that the dual problem~\eqref{eq:dual_problem} is well defined. It is envisaged that, with some additional effort, this assumption can be avoided, for instance, by reformulating the dual problem based on a regularization of the map $\bv \mapsto \bA(\bv) \colon \mathbb{R}^d \to \mathbb{R}^d$, for example, by using the techniques in~\cite{LasryLions1986}. 
\end{remark}


\section{Numerical experiments} \label{section:numerical_experiments}

We present some numerical examples to verify the theoretical error estimates presented in Section~\ref{section:error_analysis}. For simplicity, we restrict the presentation to 2D problems and consider uniformly refined meshes composed of affine quadrilaterals with uniform values of the polynomial degree~$\{ p_K \}_{K \in \mT_h}$. Throughout this section, the interior penalty parameter is fixed at~$\alpha = 10$. The nonlinear equations arising in the DG approximation are solved using an exact Newton method with a tolerance of $10^{-10}$. High-order numerical quadrature is used to integrate the terms appearing in the assembly of the associated algebraic system of equations, as well as to evaluate the error of the DG solution in various norms. 

\subsection{Example 1}
For the first numerical example, we consider the problem of Example 1 in \cite{BustinzaGatica2004}; cf. also Example 1 in \cite{HoustonEtAl2005}. Accordingly, let $\Omega = (-1, 1)^2$ with $\Gamma_\rmD = [-1, 1] \times \{ -1 \} \cup \{ 1 \} \times [-1, 1]$ and $\Gamma_\rmN = [-1, 1] \times \{ 1 \} \cup \{ -1 \} \times [-1, 1]$, and let $\bA(\bx, \nabla{u}) = \left( 2 + (1 + \abs{\nabla{u}})^{-1} \right) \mathbf{I}$, where~$\mathbf{I}$ is the~$2 \times 2$ identity matrix. The data $f$, $g_\rmD$ and $g_\rmN$ are chosen such that the solution is given by the smooth function $u(\bx) = \cos(\pi x_1 / 2) \, \cos(\pi x_2 / 2)$. We note that $\bA$ satisfies Assumption~\ref{asm:A} with~$C_\bA = 3$ and $M_\bA = 2$.  

\begin{figure}[!b] 
\psfrag{xlabel}[t][c]{$1/h$}
\psfrag{ylabel}[b][c]{$\dgnorm{u - u_{h, p}}$}
\psfrag{p = 1}[l][c]{\footnotesize \hspace{-8pt} $ p = 1$}
\psfrag{p = 2}[l][c]{\footnotesize \hspace{-8pt} $ p = 2$}
\psfrag{p = 3}[l][c]{\footnotesize \hspace{-8pt} $ p = 3$}
\psfrag{p = 4}[l][c]{\footnotesize \hspace{-8pt} $ p = 4$}
\includegraphics[width=0.45\textwidth, bb = 105 227 500 564, clip=true]{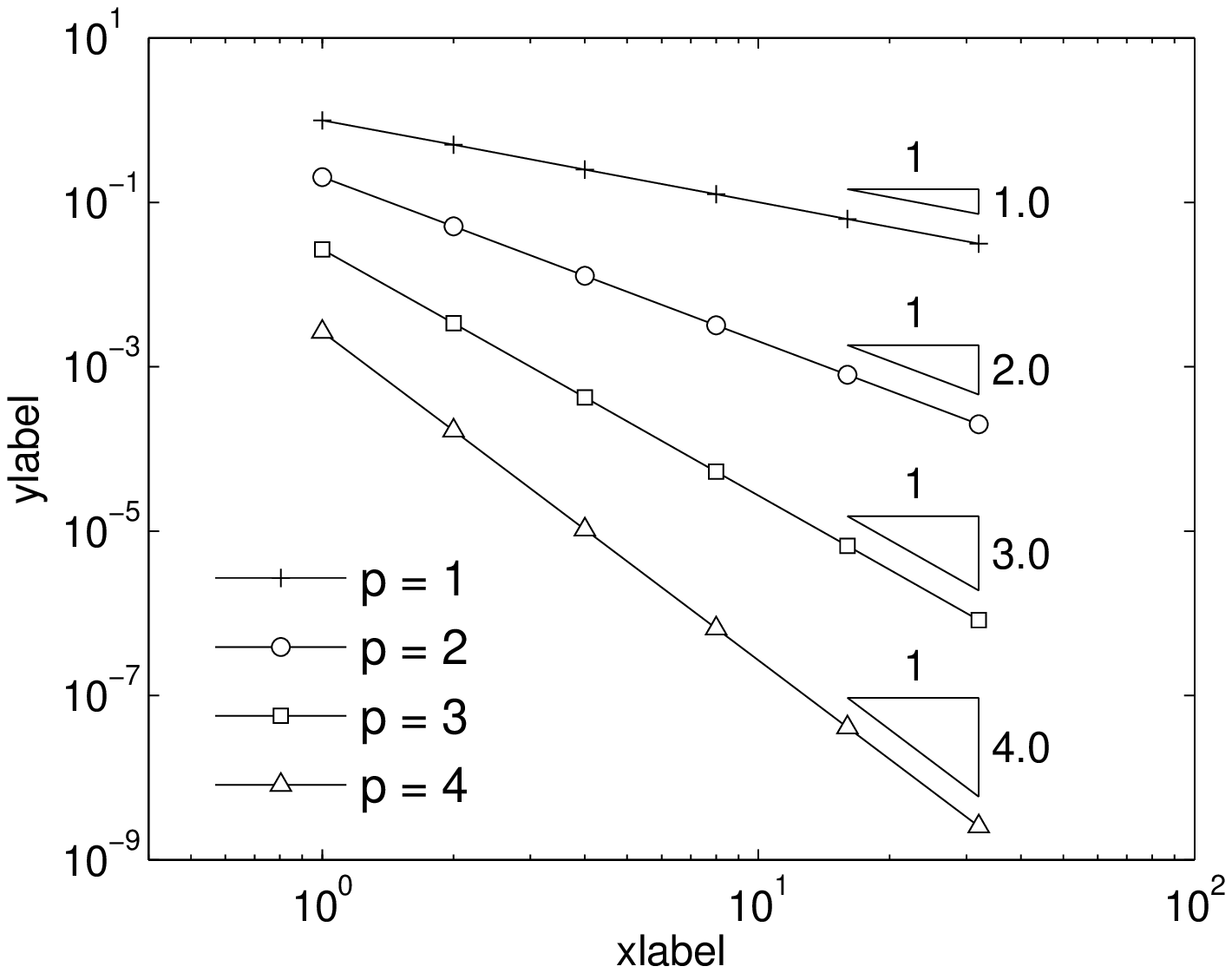}
\hfill
\includegraphics[width=0.45\textwidth, bb = 105 227 500 564, clip=true]{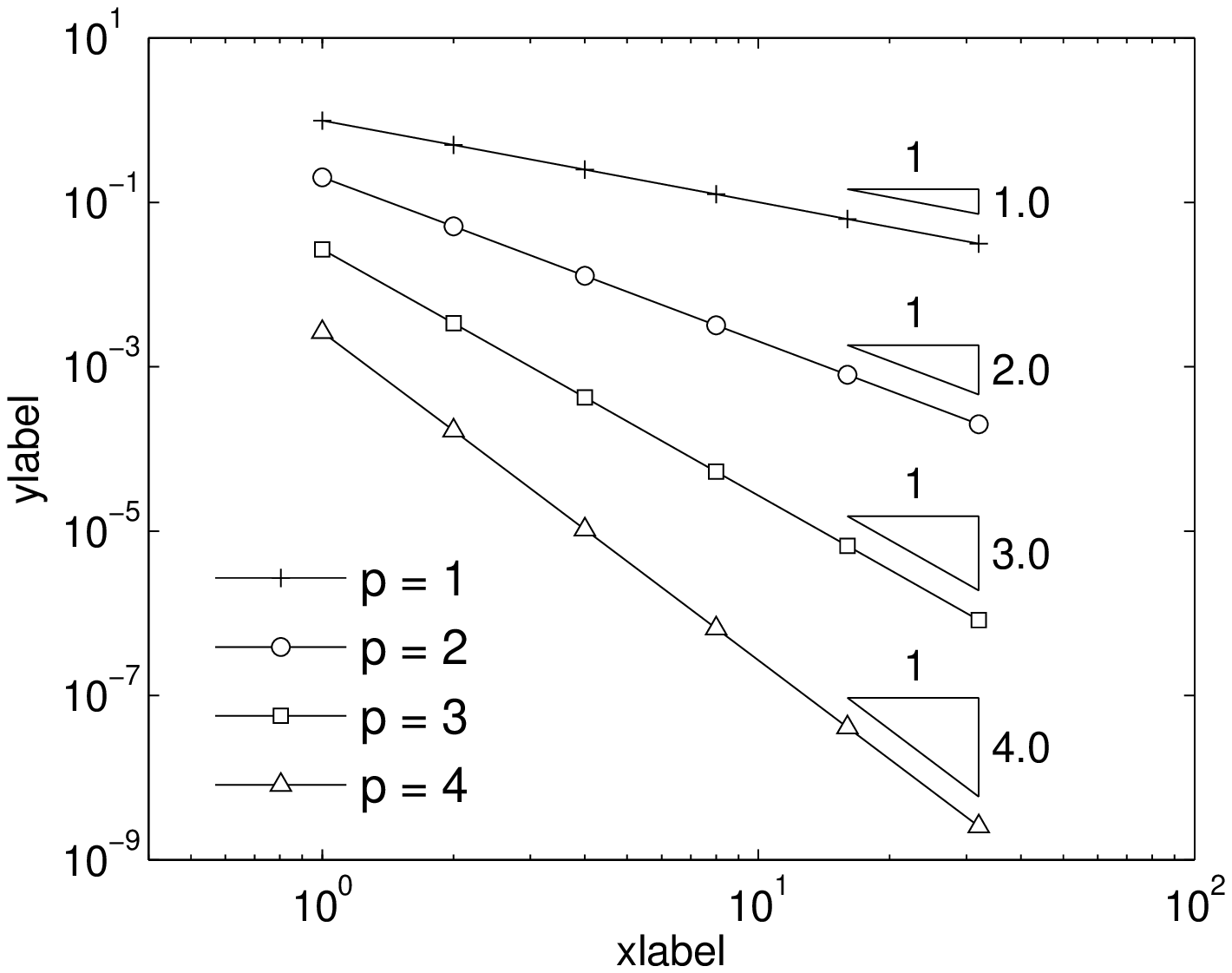}
\caption{Example 1. Convergence of $\dgnorm{u - u_{h, p}}$ with $h$-refinement for $p = 1$, $2$, $3$ and $4$. Left: $\theta = -1$. Right: $\theta = 1$.} 
\label{fig:ex1_h_eDG}
%
\psfrag{xlabel}[t][c]{$1/h$}
\psfrag{ylabel}[b][c]{$\norm{u - u_{h, p}}_{L^2(\Omega)}$}
\psfrag{p = 1}[l][c]{\footnotesize \hspace{-8pt} $ p = 1$}
\psfrag{p = 2}[l][c]{\footnotesize \hspace{-8pt} $ p = 2$}
\psfrag{p = 3}[l][c]{\footnotesize \hspace{-8pt} $ p = 3$}
\psfrag{p = 4}[l][c]{\footnotesize \hspace{-8pt} $ p = 4$}
\includegraphics[width=0.45\textwidth, bb = 105 227 500 564, clip=true]{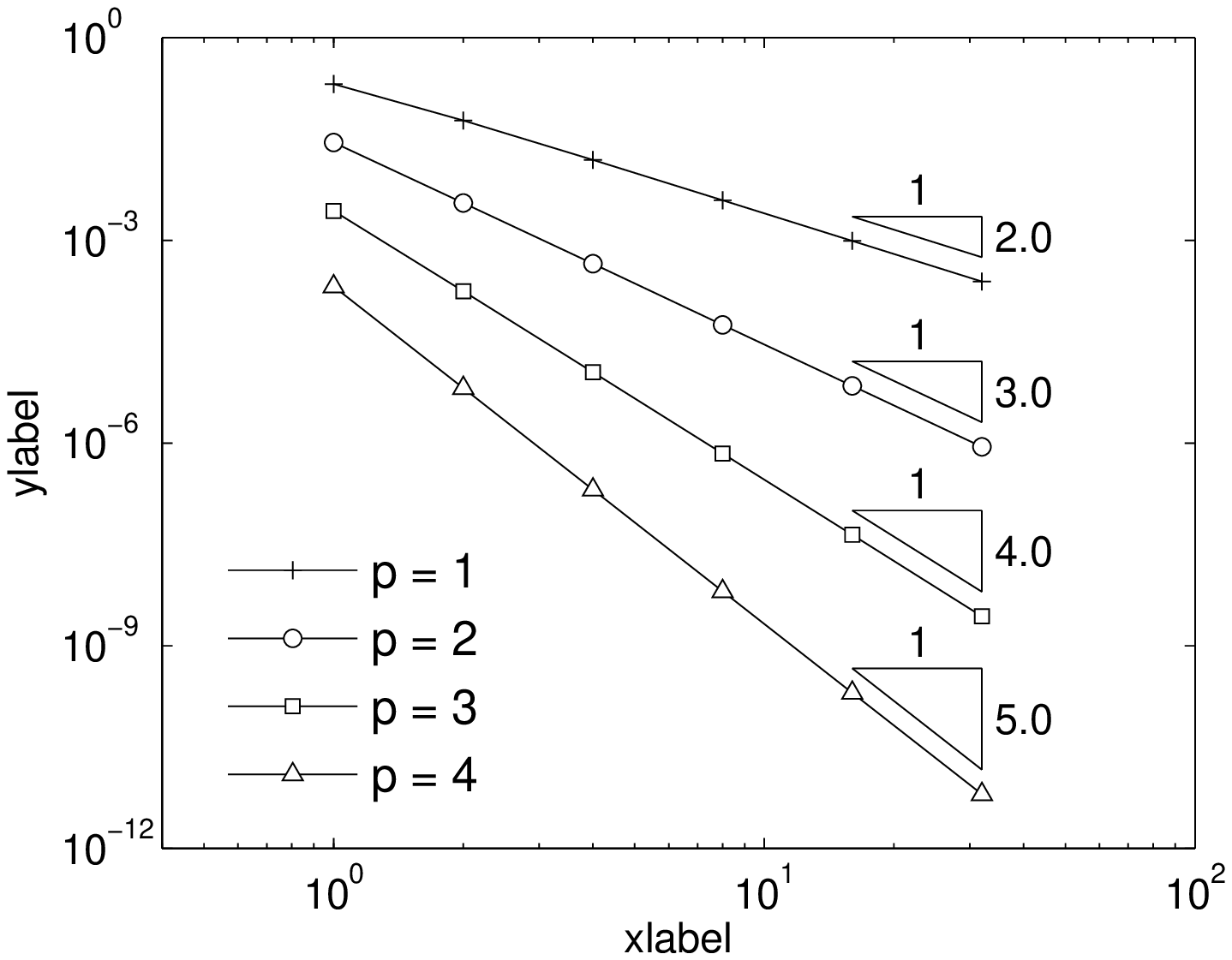}
\hfill
\includegraphics[width=0.45\textwidth, bb = 105 227 500 564, clip=true]{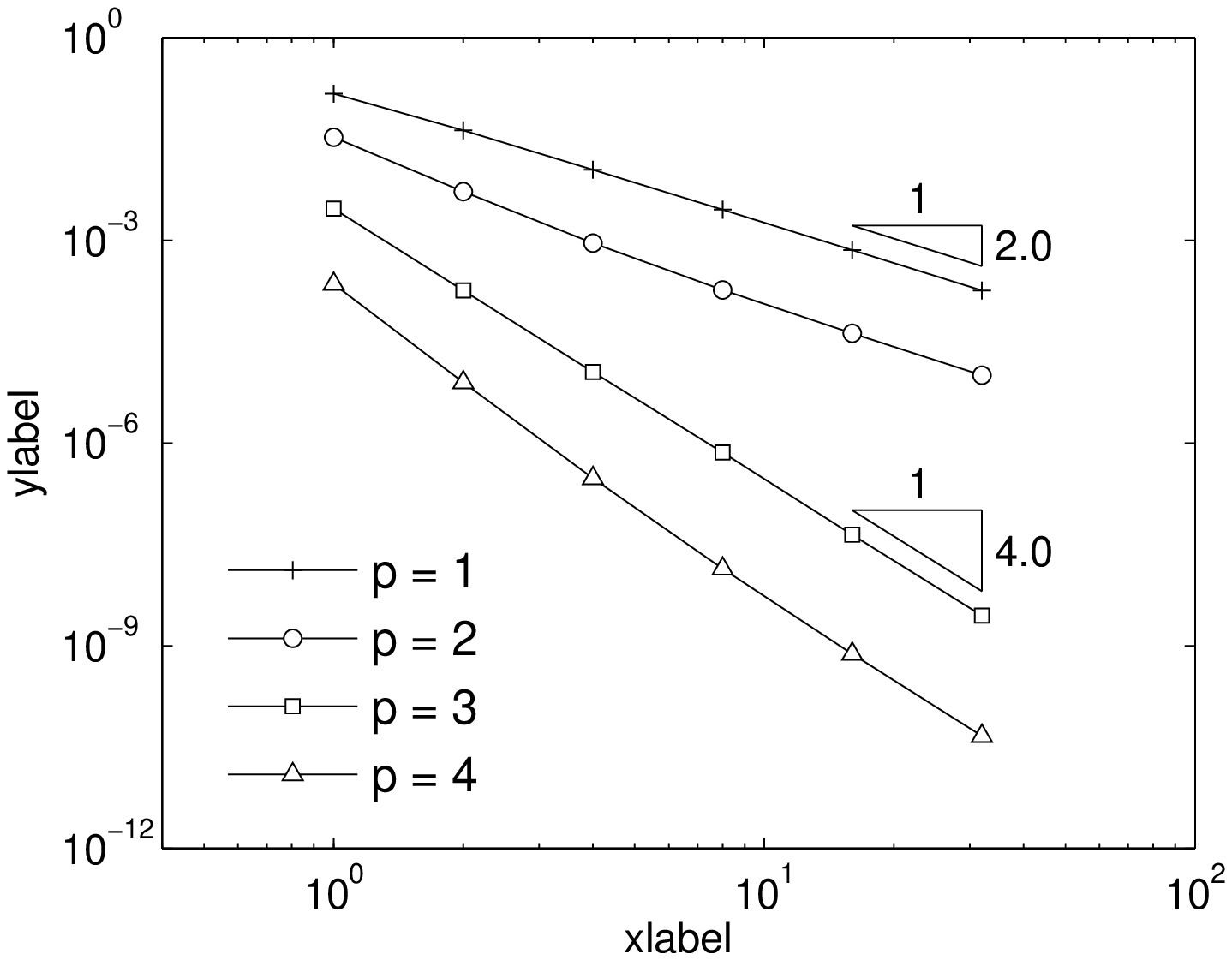}
\caption{Example 1. Convergence of $\norm{u - u_{h, p}}_{L^2(\Omega)}$ with $h$-refinement for $p = 1$, $2$, $3$ and $4$. Left: $\theta = -1$. Right: $\theta = 1$.} 
\label{fig:ex1_h_eL2}
\end{figure}

We investigate the convergence of the DG approximation~\eqref{eq:dgfem} on a sequence of successively refined meshes for different polynomial degrees. We consider two choices of the parameter $\theta$, viz. $\theta = -1$ and~$\theta = 1$. Figure~\ref{fig:ex1_h_eDG} presents the convergence of the DG-norm of the error with $h$-refinement for $p = 1$, $2$, $3$ and $4$. We observe that $\dgnorm{u - u_{h, p}}$ converges to zero, for each fixed value of~$p$, at a rate $\mathcal{O}(h^p)$ as $h \to 0$. We note that these results are in perfect agreement with the theoretical error estimate presented in Theorem~\ref{thm:dgnorm_error_estimate}, and that the computed errors are virtually indistinguishable between the two choices of the parameter~$\theta$. In Figure~\ref{fig:ex1_h_eL2}, we show the convergence of the $L^2(\Omega)$-norm of the error with $h$-refinement for $p = 1$, $2$, $3$ and $4$. Here, significant differences are observed between the two choices of~$\theta$. For~$\theta = -1$, optimal convergence rates are obtained for all values of~$p$; i.e., $\norm{u - u_{h, p}}_{L^2(\Omega)} = \mathcal{O}(h^{p+1})$ as $h \to 0$ for each fixed value of~$p$. For~$\theta = 1$ on the other hand, we see that $\norm{u - u_{h, p}}_{L^2(\Omega)}$ behaves like $\mathcal{O}(h^{p+1})$ as $h \to 0$ for odd values of $p$, and like $\mathcal{O}(h^p)$ as $h \to 0$ for even values of $p$. This suboptimal convergence behavior for~$\theta = 1$ is attributable to a lack of dual consistency; cf. Lemma~\ref{lem:dual_consistency}. The obtained convergence rates for $\norm{u - u_{h, p}}_{L^2(\Omega)}$ are in agreement with the theoretical error estimates presented in Corollary~\ref{crl:l2norm_error_estimate}. Comparing the current result to the results reported for the same example in \cite{HoustonEtAl2005}, we note that presented DG method with $\theta = -1$ shows improved convergence behavior with respect to the error in the $L^2(\Omega)$-norm.

\subsection{Example 2}
In the second example, we consider a problem with a non-smooth solution. Let $\Omega = (-1, 1)^2$ with $\Gamma_\rmD = \partial \Omega$, and $\bA(\bx, \nabla{u}) = ( 1 + \mathrm{e}^{-\abs{\nabla{u}}^2} ) \mathbf{I}$, where~$\mathbf{I}$ denotes again the~$2 \times 2$ identity matrix. It is easy to verify that Assumption~\ref{asm:A} is satisfied with~$C_\bA = 1$ and $M_\bA = 1 - \sqrt{2/\mathrm{e}}$. The data $f$ and $g_\rmD$ are chosen such that the solution is given by $u(\bx) = \abs{\bx}^3$. We note that the solution features a singularity at the point $(0, 0)$, and that $u \in H^{4 - \epsilon}(\Omega)$ for arbitrary small $\epsilon > 0$.

We investigate the convergence behavior with $p$-refinement for the two meshes displayed in Figure~\ref{fig:ex2_meshes}. In Tables~\ref{table:ex2_mesh_a} and~\ref{table:ex2_mesh_b}, we show the convergence of the DG-norm of the error and the $L^2(\Omega)$-norm for~$p = 1$, $2$, \dots, $24$, and $\theta = -1$, grouped in odd and even values of $p$. For mesh (a), we observe that $\dgnorm{u - u_{h, p}}$ converges at a rate of almost $\mathcal{O}(p^{-6})$ as $p \to \infty$, and that $\norm{u - u_{h, p}}_{L^2(\Omega)}$ converges at a rate of approximately $\mathcal{O}(p^{-15/2})$. Comparing with the theoretical error estimates of Theorem~\ref{thm:dgnorm_error_estimate} and Corollary~\ref{crl:l2norm_error_estimate}, we note that these convergence rates are more than twice the predicted rate. Indeed, since $u \in H^{4 - \epsilon}$ for any $\epsilon > 0$, the expected convergence rates are $\mathcal{O}(p^{-5/2+\epsilon})$ for $\dgnorm{u - u_{h, p}}$ and $\mathcal{O}(p^{-3+\epsilon})$ for $\norm{u - u_{h, p}}_{L^2(\Omega)}$. This order-doubling convergence behavior is attributable to the fact that the singularity in $u$ at the point $(0, 0)$ coincides with a vertex of mesh (a). In the presence of such corner singularities, it is possible to establish \emph{a priori} error estimates that reflect this order-doubling phenomenon by using approximation results in terms of weighted Sobolev norms; cf., for example, \cite[Remark 3.8]{HoustonEtAl2002}. For mesh (b), on the other hand, the singularity in $u$ lies in the interior of an element rather than at a vertex. Here, we see that the $p$-convergence rates approach the theoretical convergence rates predicted by Theorem~\ref{thm:dgnorm_error_estimate} and Corollary~\ref{crl:l2norm_error_estimate}. Indeed, it is found that $\dgnorm{u - u_{h, p}}$ and $\norm{u - u_{h, p}}_{L^2(\Omega)}$ both behave like $\mathcal{O}(p^{-3})$ as $p \to \infty$. For $\dgnorm{u - u_{h, p}}$, this constitutes a slight improvement of the theoretical convergence rate, by half an order in $p$, while for $\norm{u - u_{h, p}}_{L^2(\Omega)}$ the convergence rate is in perfect agreement. We end this example by stating that the results for $\theta = 1$ are almost identical.

\begin{figure}[h]
\psfrag{p0}[c][r]{\footnotesize $(-1, -1)$}
\psfrag{p1}[c][l]{\footnotesize $( 1, -1)$}
\psfrag{p2}[b][l]{\footnotesize $( 1,  1)$}
\psfrag{p3}[b][r]{\footnotesize $(-1,  1)$}
\centering
\begin{minipage}[h]{0.3\textwidth} \centering
\includegraphics[width=1.0\textwidth]{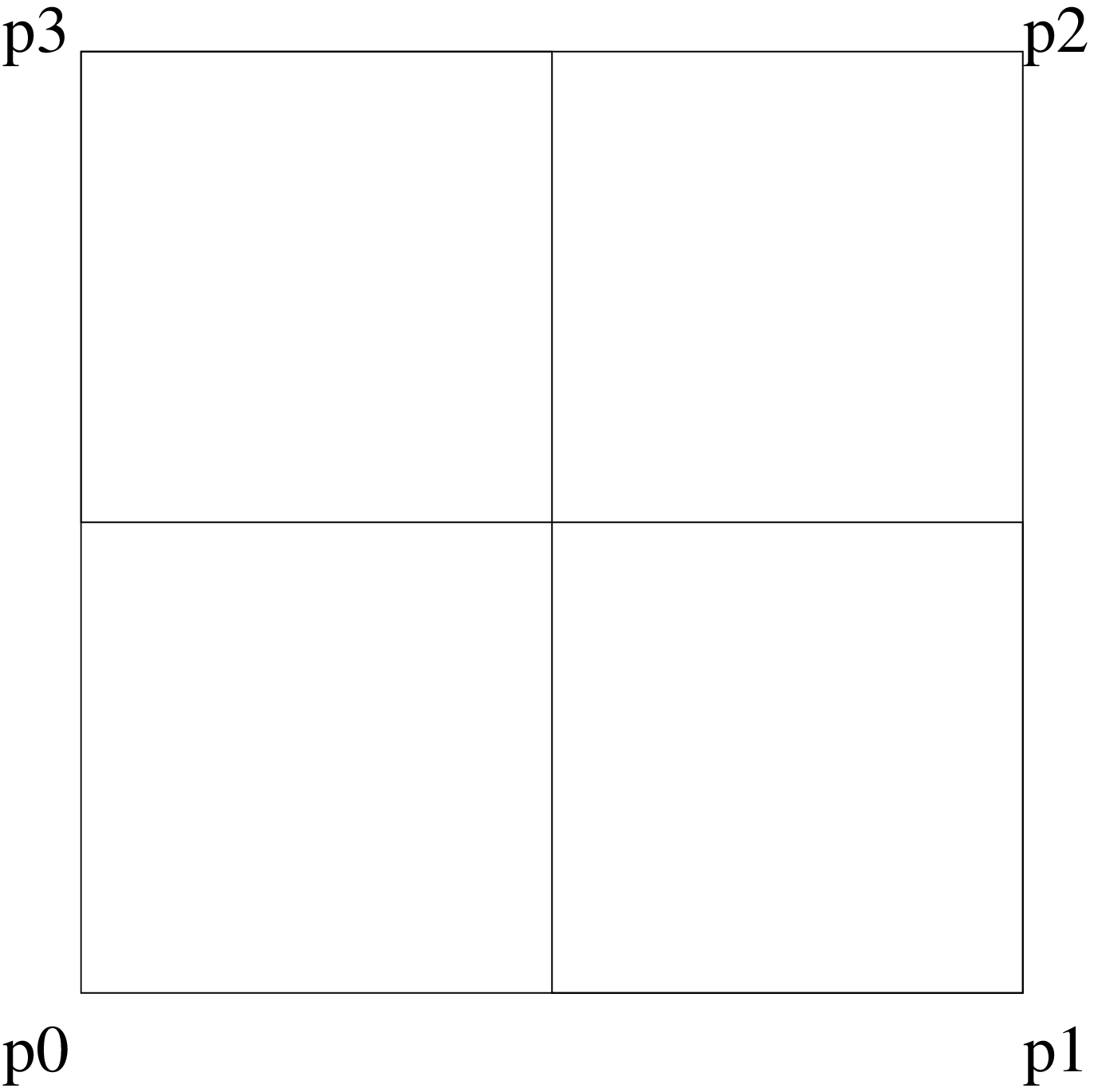} \\
{\footnotesize (a)}
\end{minipage}
\hspace{1cm}
\begin{minipage}[h]{0.3\textwidth} \centering
\includegraphics[width=1.0\textwidth]{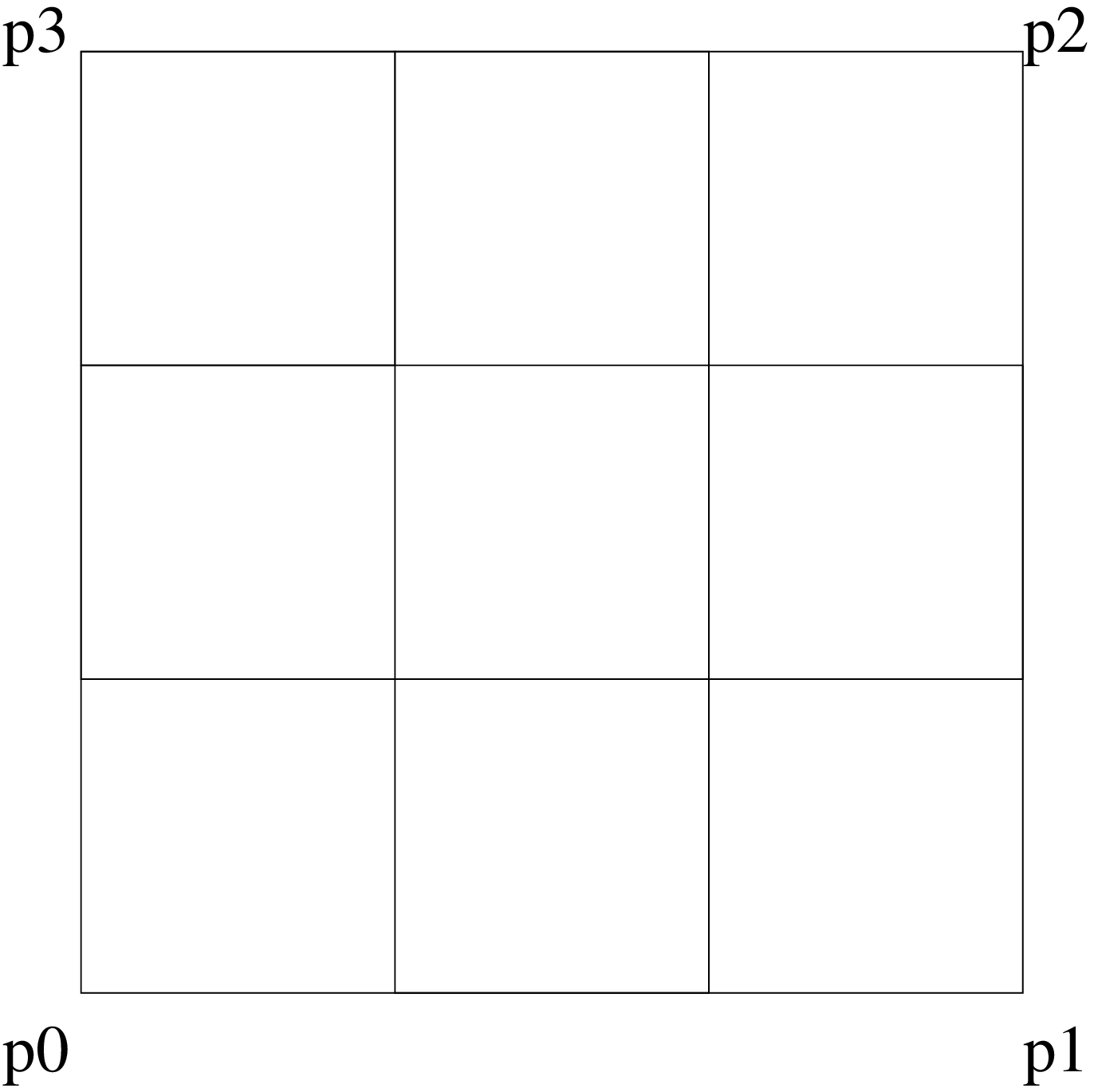} \\
{\footnotesize (b)}
\end{minipage}
\caption{The two meshes considered for Example 2.} 
\label{fig:ex2_meshes}
\end{figure}

\begin{table}[htb]
\setlength{\tabcolsep}{5pt}
\centering
\begin{minipage}[h]{0.45\textwidth} \centering \small
\begin{tabular}{llclc}
\hline 
$p$ & \multicolumn{2}{l}{$\dgnorm{u - u_{h, p}}$} & \multicolumn{2}{l}{$\norm{u - u_{h, p}}_{L^2(\Omega)}$} \\
\hline
1  & 3.11E+00 & \textemdash & 4.46E-01 & \textemdash \\ 
3  & 4.09E-02 &  (3.94) & 3.30E-03 &  (4.47) \\
5  & 2.17E-03 &  (5.75) & 1.51E-04 &  (6.03) \\
7  & 2.84E-04 &  (6.05) & 1.35E-05 &  (7.18) \\
9  & 6.38E-05 &  (5.94) & 1.97E-06 &  (7.66) \\
11 & 1.96E-05 &  (5.89) & 4.51E-07 &  (7.34) \\
13 & 7.32E-06 &  (5.88) & 1.31E-07 &  (7.43) \\
15 & 3.16E-06 &  (5.88) & 4.50E-08 &  (7.44) \\
17 & 1.51E-06 &  (5.88) & 1.76E-08 &  (7.50) \\
19 & 7.86E-07 &  (5.88) & 7.64E-09 &  (7.51) \\
21 & 4.36E-07 &  (5.88) & 3.59E-09 &  (7.55) \\
23 & 2.55E-07 &  (5.89) & 1.80E-09 &  (7.57) \\
\hline
\end{tabular}
\end{minipage}
\hfill
\begin{minipage}[h]{0.45\textwidth} \centering \small
\begin{tabular}{llclc}
\hline
$p$ & \multicolumn{2}{l}{$\dgnorm{u - u_{h, p}}$} & \multicolumn{2}{l}{$\norm{u - u_{h, p}}_{L^2(\Omega)}$} \\
\hline
2  & 5.74E-01 &  \textemdash & 8.60E-02 &  \textemdash \\  
4  & 8.72E-03 &  (6.04) & 6.74E-04 &  (7.00) \\ 
6  & 7.12E-04 &  (6.18) & 3.71E-05 &  (7.15) \\
8  & 1.28E-04 &  (5.96) & 5.00E-06 &  (6.97) \\
10 & 3.43E-05 &  (5.91) & 9.28E-07 &  (7.55) \\
12 & 1.17E-05 &  (5.89) & 2.37E-07 &  (7.49) \\
14 & 4.73E-06 &  (5.88) & 7.52E-08 &  (7.45) \\
16 & 2.16E-06 &  (5.88) & 2.78E-08 &  (7.46) \\
18 & 1.08E-06 &  (5.88) & 1.15E-08 &  (7.49) \\
20 & 5.81E-07 &  (5.88) & 5.19E-09 &  (7.54) \\
22 & 3.32E-07 &  (5.88) & 2.53E-09 &  (7.56) \\
24 & 1.99E-07 &  (5.89) & 1.30E-09 &  (7.59) \\
\hline
\end{tabular}
\end{minipage}
\vspace{6pt}
\caption{Example 2. Convergence of $\dgnorm{u - u_{h, p}}$ and $\norm{u - u_{h, p}}_{L^2(\Omega)}$ with $p$-refinement for mesh (a) and $\theta = - 1$. The results are grouped in odd and even values of $p$. The quantities in brackets indicate the $p$-convergence rates.} \label{table:ex2_mesh_a}
\vspace{24pt}
%
\begin{minipage}[h]{0.45\textwidth} \centering \small
\begin{tabular}{llclc}
\hline 
$p$ & \multicolumn{2}{l}{$\dgnorm{u - u_{h, p}}$} & \multicolumn{2}{l}{$\norm{u - u_{h, p}}_{L^2(\Omega)}$} \\
\hline
1  & 2.07E+00 & \textemdash & 2.31E-01 & \textemdash \\
3  & 3.39E-02 &  (3.74) & 2.22E-03 &  (4.23) \\
5  & 3.42E-03 &  (4.49) & 1.67E-04 &  (5.07) \\
7  & 1.03E-03 &  (3.55) & 4.39E-05 &  (3.96) \\
9  & 4.49E-04 &  (3.32) & 1.81E-05 &  (3.53) \\
11 & 2.35E-04 &  (3.23) & 9.29E-06 &  (3.32) \\
13 & 1.38E-04 &  (3.17) & 5.44E-06 &  (3.20) \\
15 & 8.82E-05 &  (3.14) & 3.48E-06 &  (3.13) \\
17 & 5.97E-05 &  (3.11) & 2.36E-06 &  (3.09) \\
19 & 4.23E-05 &  (3.10) & 1.68E-06 &  (3.06) \\
21 & 3.11E-05 &  (3.08) & 1.24E-06 &  (3.04) \\
23 & 2.35E-05 &  (3.08) & 9.42E-07 &  (3.02) \\
\hline
\end{tabular}
\end{minipage}
\hfill
\begin{minipage}[h]{0.45\textwidth} \centering \small
\begin{tabular}{llclc}
\hline
$p$ & \multicolumn{2}{l}{$\dgnorm{u - u_{h, p}}$} & \multicolumn{2}{l}{$\norm{u - u_{h, p}}_{L^2(\Omega)}$} \\
\hline
2  & 2.21E-01 & \textemdash & 2.12E-02 & \textemdash \\
4  & 3.63E-03 &  (5.93) & 4.62E-04 &  (5.52) \\
6  & 1.09E-03 &  (2.96) & 1.50E-04 &  (2.78) \\
8  & 4.75E-04 &  (2.90) & 6.64E-05 &  (2.83) \\
10 & 2.49E-04 &  (2.89) & 3.50E-05 &  (2.86) \\
12 & 1.47E-04 &  (2.90) & 2.07E-05 &  (2.89) \\
14 & 9.38E-05 &  (2.91) & 1.32E-05 &  (2.90) \\
16 & 6.36E-05 &  (2.92) & 8.97E-06 &  (2.91) \\
18 & 4.51E-05 &  (2.92) & 6.36E-06 &  (2.92) \\
20 & 3.31E-05 &  (2.93) & 4.67E-06 &  (2.93) \\
22 & 2.50E-05 &  (2.93) & 3.53E-06 &  (2.94) \\
24 & 1.94E-05 &  (2.94) & 2.73E-06 &  (2.94) \\
\hline
\end{tabular}
\end{minipage}
\vspace{6pt}
\caption{Example 2. Convergence of $\dgnorm{u - u_{h, p}}$ and $\norm{u - u_{h, p}}_{L^2(\Omega)}$ with $p$-refinement for mesh (b) and $\theta = - 1$. The results are grouped in odd and even values of $p$. The quantities in brackets indicate the $p$-convergence rates.}
\label{table:ex2_mesh_b}
\end{table}

\subsection{Example 3}
In the third and final example, we consider a case not fully covered by our theory. We consider the solution of the $p(\bx)$-Laplace equation with $\bA(\bx, \nabla{u}) = \abs{\nabla{u}}^{p(\bx)-2} \mathbf{I}$, where $p(\bx) = 4 - \abs{\bx}^2$. Note that $\bA$ does not comply with Assumption~\ref{asm:A} for $\abs{\bx} < 1$. The problem is posed on the $L$-shaped domain $\Omega = (-1, 1)^2 \setminus [0, 1) \times (-1, 0]$ with $\Gamma_\rmN = [-1, 1] \times \{ 1 \} \cup \{ -1 \} \times [-1, 1]$ and $\Gamma_\rmD = \partial \Omega \setminus \Gamma_\rmN$. The data $f$, $g_\rmD$ and $g_\rmN$ are chosen such that the solution is given by the smooth function $u(\bx) = x_1 \, \mathrm{e}^{x_1 \, x_2}$.

In Figure~\ref{fig:ex3_h_eDG}, we show the convergence of the DG-norm of the error with $h$-refinement for $p = 1$, $2$, $3$, $4$ and $\theta = -1$, $1$. As in Example 1, we observe that $\dgnorm{u - u_{h, p}}$ converges to zero, for each fixed value of~$p$, at a rate $\mathcal{O}(h^p)$ as $h \to 0$. Note that this is in perfect agreement with the theoretical error estimate presented in Theorem~\ref{thm:dgnorm_error_estimate}, even though the underlying Assumption~\ref{asm:A} is not met. Also note that the results are virtually distinguishable between the two choices of the parameter $\theta$. In Figure~\ref{fig:ex3_h_eL2}, we present the convergence of the $L^2(\Omega)$-norm with $h$-refinement for $p = 1$, $2$, $3$, $4$ and $\theta = -1$, $1$. Here, as in Example 1, significant differences are observed between the two values of~$\theta$. For~$\theta = -1$, optimal convergence rates are obtained for all values of~$p$; i.e., $\norm{u - u_{h, p}}_{L^2(\Omega)} = \mathcal{O}(h^{p+1})$ as $h \to 0$ for each fixed value of~$p$. For~$\theta = 1$ on the other hand, we see that $\norm{u - u_{h, p}}_{L^2(\Omega)}$ behaves like $\mathcal{O}(h^{p+1})$ as $h \to 0$ for odd values of $p$, and like $\mathcal{O}(h^p)$ as $h \to 0$ for even values of $p$. The convergence behavior for $\norm{u - u_{h, p}}_{L^2(\Omega)}$ is very similar to that seen in Example 1 and agrees well with the theoretical error estimates of Corollary~\ref{crl:l2norm_error_estimate}.

\begin{figure}[h]
\psfrag{xlabel}[t][c]{$1/h$}
\psfrag{ylabel}[b][c]{$\dgnorm{u - u_{h, p}}$}
\psfrag{p = 1}[l][c]{\footnotesize \hspace{-8pt} $ p = 1$}
\psfrag{p = 2}[l][c]{\footnotesize \hspace{-8pt} $ p = 2$}
\psfrag{p = 3}[l][c]{\footnotesize \hspace{-8pt} $ p = 3$}
\psfrag{p = 4}[l][c]{\footnotesize \hspace{-8pt} $ p = 4$}
\includegraphics[width=0.45\textwidth, bb = 105 227 500 564, clip=true]{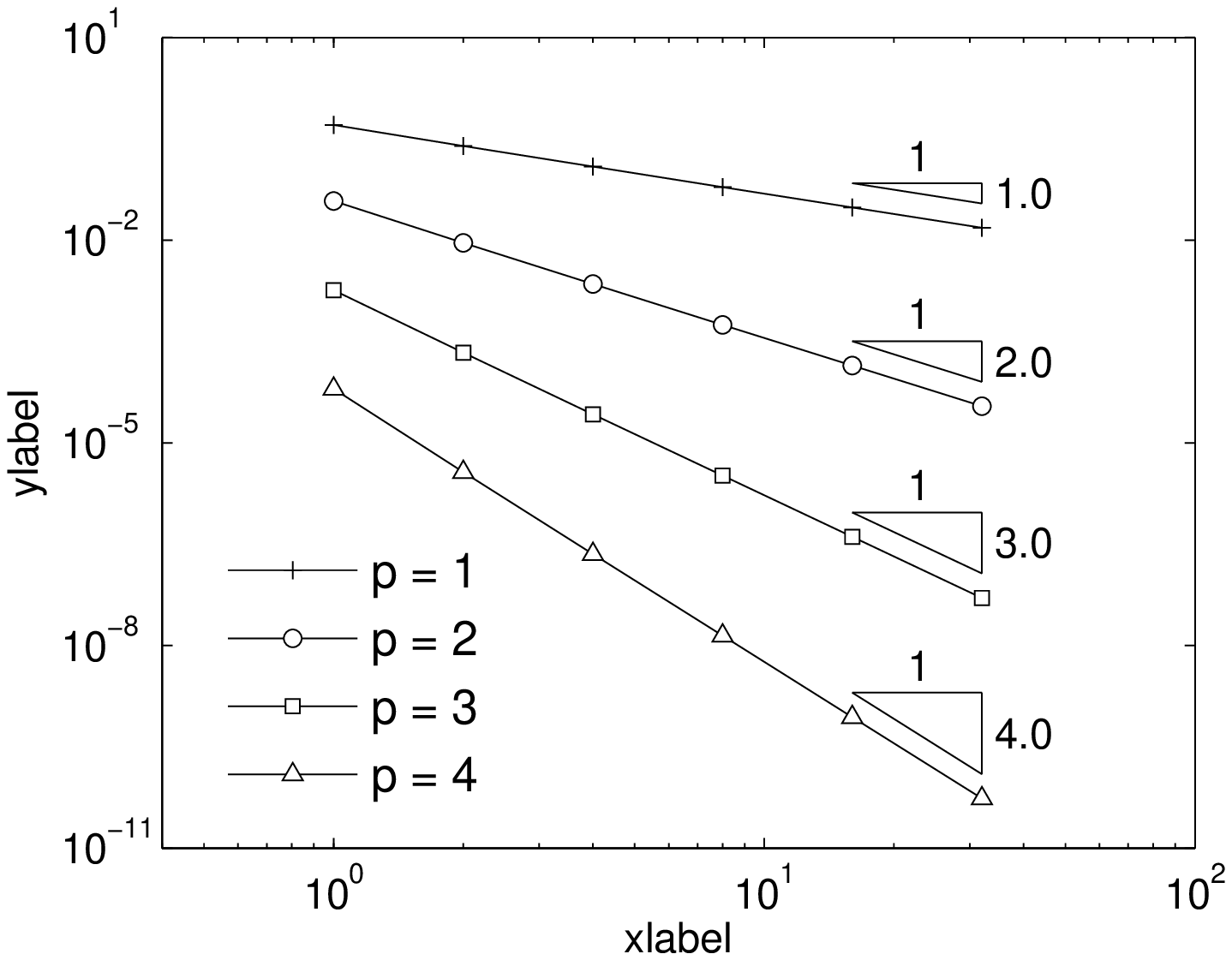}
\hfill
\includegraphics[width=0.45\textwidth, bb = 105 227 500 564, clip=true]{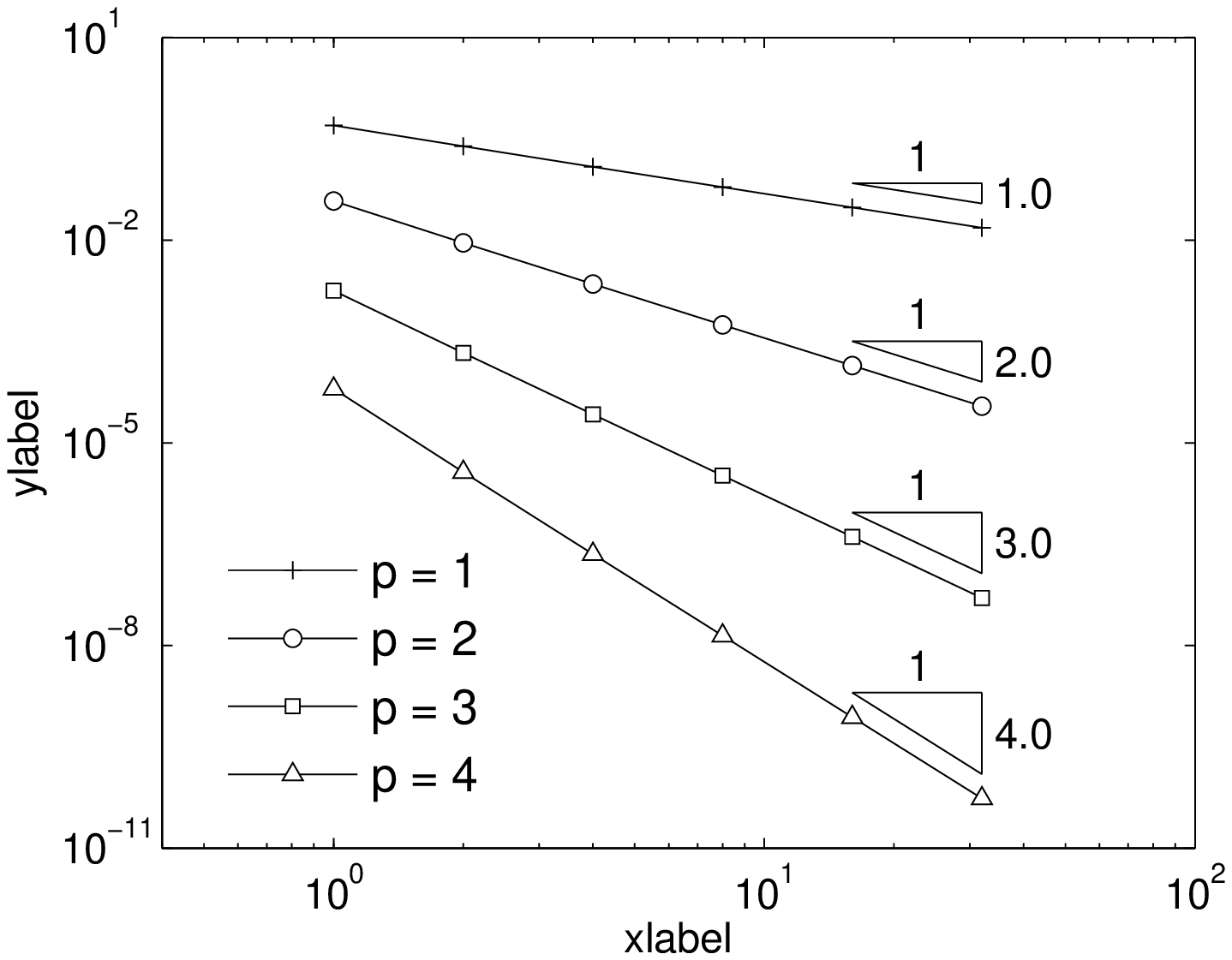}
\caption{Example 3. Convergence of $\dgnorm{u - u_{h, p}}$ with $h$-refinement for $p = 1$, $2$, $3$ and $4$. Left: $\theta = -1$. Right: $\theta = 1$.}
\label{fig:ex3_h_eDG}
%
\psfrag{xlabel}[t][c]{$1/h$}
\psfrag{ylabel}[b][c]{$\norm{u - u_{h, p}}_{L^2(\Omega)}$}
\psfrag{p = 1}[l][c]{\footnotesize \hspace{-8pt} $ p = 1$}
\psfrag{p = 2}[l][c]{\footnotesize \hspace{-8pt} $ p = 2$}
\psfrag{p = 3}[l][c]{\footnotesize \hspace{-8pt} $ p = 3$}
\psfrag{p = 4}[l][c]{\footnotesize \hspace{-8pt} $ p = 4$}
\includegraphics[width=0.45\textwidth, bb = 105 227 500 564, clip=true]{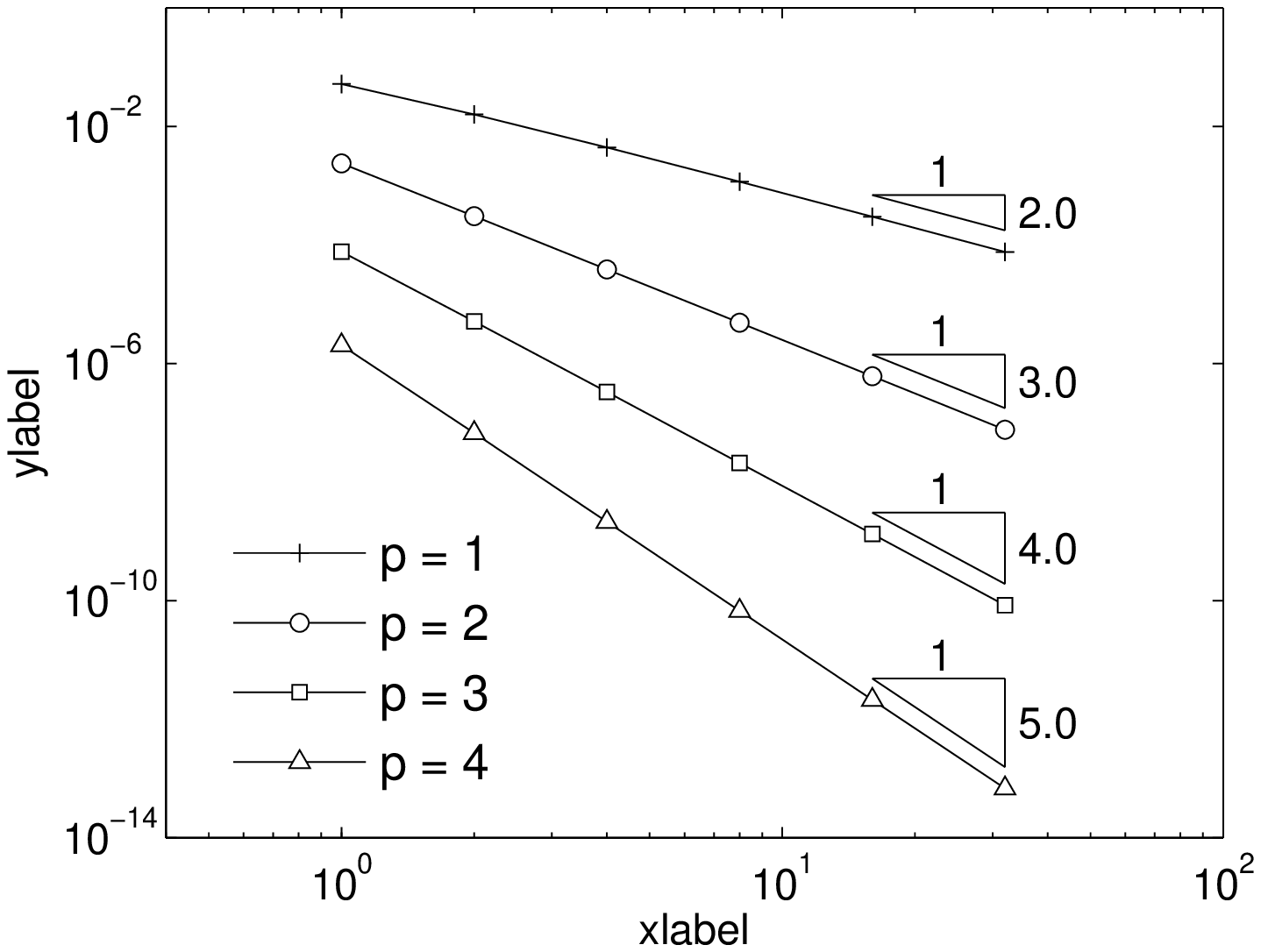}
\hfill
\includegraphics[width=0.45\textwidth, bb = 105 227 500 564, clip=true]{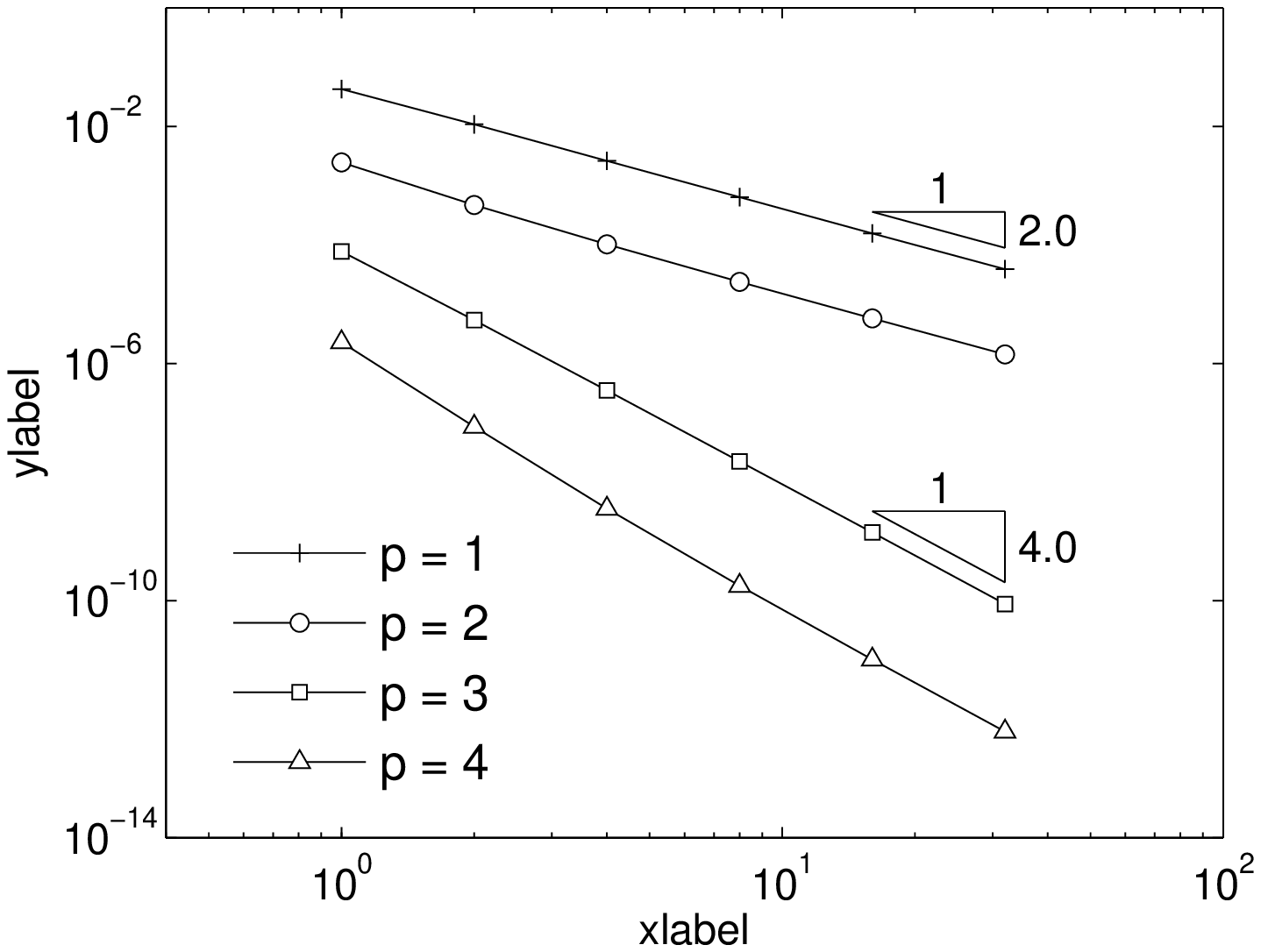}
\caption{Example 3. Convergence of $\norm{u - u_{h, p}}_{L^2(\Omega)}$ with $h$-refinement for $p = 1$, $2$, $3$ and $4$. Left: $\theta = -1$. Right: $\theta = 1$.}
\label{fig:ex3_h_eL2}
\end{figure}


\section*{Acknowledgement}

The work presented in this paper was completed while the author was a Ph.D. student at the Delft University of Technology working under the supervision of Dr.~S.~J.~Hulshoff, for whose guidance and support the author is most grateful.


\appendix

\section{Nonlinear inf-sup theory}

We include some auxiliary results regarding the well-posedness of nonlinear variational problems. Let~$U$ be a real Banach space equipped with the norm~$\norm{\cdot}_U$, and let~$V$ be a real reflexive Banach space equipped with the norm~$\norm{\cdot}_V$. We denote by~$U'$ and~$V'$ the respective dual spaces, equiped with the norms
\begin{equation*}
\norm{f}_{U'} = \sup_{u \in U \setminus \{ 0 \}} \frac{ \langle f, v \rangle_{U', U} }{\norm{u}_U} , \qquad \norm{g}_{V'} = \sup_{v \in V \setminus \{ 0 \}} \frac{ \langle g, v \rangle_{V', V} }{\norm{v}_V} ,
\end{equation*}
where $\langle \cdot, \cdot \rangle_{U', U}$ and $\langle \cdot, \cdot \rangle_{V', V}$ are the duality pairings between $U'$ and $U$, and $V'$ and $V$, respectively.

The first result that we present constitutes a nonlinear extension of the classical well-posedness result of Banach, Ne\v{c}as and Babu\v{s}ka; cf., for example, \cite[Theorem~1.1]{PietroErn2011}. The statement of the theorem and parts of its proof are adopted from~\cite[Appendix~A]{BrummelenBorst2005}, where the theorem is presented in a Hilbert space setting. We note that the theorem generalizes some other results from the literature; see, for example, \cite[Theorem~25.B]{Zeidler1990}.
\begin{theorem}[$\inf$-$\sup$ conditions] \label{thm:nonlinear_infsup}
Let~$a \colon U \times V \to \mathbb{R}$ be a semilinear form, such that
\begin{equation}
a(w_1; v) - a(w_2; v) \leq C_a \, \norm{w_1 - w_2}_U \ \norm{v}_V \quad \forall w_1. w_2 \in U, \ \forall v \in V \label{eq:BNB_continuity}
\end{equation}
for some constant~$C_a > 0$. Then, the variational problem
\begin{equation}
u \in U : \quad a(u; v) = f(v) \qquad \forall v \in V \label{eq:BNB_problem}
\end{equation}
admits a unique solution~$u \in U$ for every~$f \in V'$ if and only if
\begin{equation}
\exists M_a > 0 : \quad \inf_{\substack{w_1, w_2 \in U \\ w_1 \neq w_2}} \ \sup_{v \in V \setminus \{0\}} \ \frac{a(w_1; v) - a(w_2; v)}{\norm{w_1 - w_2}_U \norm{v}_V} \geq M_a \, , \label{eq:BNB_1}
\end{equation}
\begin{equation}
\sup_{w \in U} a(w; v) > 0 \quad \forall v \in V \setminus \{0\} \, . \label{eq:BNB_2}
\end{equation}
Moreover, for any~$g \in V' \setminus \{f\}$ and corresponding~$\tilde{u} \in U$ such that~$a(\tilde{u}; v) = g(v)$ for all~$v \in V$, we have the following \emph{a~priori} estimate:
\begin{equation}
\norm{u - \tilde{u}}_U \leq \frac{1}{M_a} \norm{f - g}_{V'} . \label{eq:BNB_apriori}
\end{equation}
\end{theorem}
\begin{proof}
The proof proceeds in a similar manner as for the linear setting; cf., for example, \cite{Schwab1998}. For any fixed~$w \in U$, consider the linear functional $\phi_w \colon V \to \mathbb{R}$ of the form $v \mapsto \phi_w(v) := a(w; v)$ for all $v \in V$. By virtue of~\eqref{eq:BNB_continuity} with~$w_1 = w$ and~$w_2 = 0$, we have that
\begin{equation*}
\norm{\phi_w}_{V'} = \sup_{v \in V \setminus \{0\}} \frac{\abs{\phi_w(v)}}{\norm{v}_V} = \sup_{v \in V \setminus \{0\}} \frac{\abs{a(w; v)}}{\norm{v}_V} \leq C_a \norm{w}_U .
\end{equation*}
Hence, $\phi_w \in V'$. Now, let $A \colon U \to V'$ such that $w \mapsto A(w) := \phi_w$ for all~$w \in U$. The variational problem \eqref{eq:BNB_problem} is then equivalent to finding $u \in U$ such that~$A(u) = f$ in~$V'$. The existence and uniqueness of a solution $u \in U$ is ensured if the operator $A \colon U \to V'$ is injective and surjective.

Injectivity of $A$ is established by verifying that $A(w_1) = A(w_2)$ implies $w_1 = w_2$. By virtue of~\eqref{eq:BNB_1}, we have that, for all~$w_1, w_2 \in U$,
\begin{align}
\norm{A(w_1) - A(w_2)}_{V'} \, 
= \ 
& \sup_{v \in V \setminus \{0\}} \frac{\langle A(w_1) - A(w_2) , v \rangle_{V', V}}{\norm{v}_V} \nonumber
\\ = \
& \sup_{v \in V \setminus \{0\}} \frac{a(w_1; v) - a(w_2; v)}{\norm{v}_V} \nonumber
\\ \geq \
& M_a \norm{w_1 - w_2}_U . \label{eq:BNB_3}
\end{align}
Consequently, $\norm{A(w_1) - A(w_2)}_{V'} = 0$ implies~$\norm{w_1 - w_2}_U = 0$. Hence, $A$ is injective.

Surjectivity of $A$ is established by verifying that the range of $A$, hereafter denoted by $\mathrm{Im}(A)$, coincides with $V'$. This is equivalent to showing that $\mathrm{Im}(A)$ is closed in $V'$, and that its ortogonal complement in $V'$ is empty. To this end, let $\{w_n\}_{n=0}^\infty$ be some sequence in $U$ such that $\{ A(w_n) \}_{n=0}^\infty$ is a Cauchy sequence in $V'$. Then, from~\eqref{eq:BNB_3}, it follows that~$\{w_n\}_{n=0}^\infty$ is Cauchy in $U$. Let $w$ be its limit. On account of~\eqref{eq:BNB_continuity}, we have that $A(w_n) \to A(w)$ as~$n \to \infty$; indeed, for~$n \to \infty$,
\begin{align*}
\norm{A(w) - A(w_n)}_{V'} \,
= \
& \sup_{v \in V \setminus \{0\}} \frac{\langle A(w) - A(w_n) , v\rangle_{V', V}}{\norm{v}_V}
\\ = \
& \sup_{v \in V \setminus \{0\}} \frac{a(w; v) - a(w_n; v)}{\norm{v}_V} \nonumber
\\ \leq \
& C_a \norm{w - w_n}_U \to 0 .
\end{align*}
This in turn implies that~$A(w) \in \mathrm{Im}(A)$ and, thus, that~$\mathrm{Im}(A)$ is closed. It remains to show that the orthogonal complement of $A$ in $V'$ is empty. Let us argue by contradiction by supposing that $\mathrm{Im}(A) \subsetneq V'$. Then, by the Hahn-Banach theorem in the form of~\cite[Proposition 3]{Zeidler1995}, there exists a~$v_0 \in V''$ such that~$\langle A(w) , v_0 \rangle_{V', V''} = 0$ for every~$w \in U$. Since $V$ is reflexive, we can identify $V''$ with $V$ so that $v_0 \in V$. Accordingly, we have
\begin{align*}
0 = \langle A(w) , v_0 \rangle_{V', V} = a(w; v_0) \quad \forall w \in U ,
\end{align*}
which is in contradiction to~\eqref{eq:BNB_2}. This implies that $V' \setminus \mathrm{Im}(A) = \emptyset$ and, therefore, $\mathrm{Im}(A) \equiv V'$. Hence, $A$ is surjective.

Based on the above, we conclude that~\eqref{eq:BNB_problem} has a unique solution~$u \in U$ for every~$f \in V'$ whenever~\eqref{eq:BNB_1} and~\eqref{eq:BNB_2} hold. The \emph{a~priori} estimate~\eqref{eq:BNB_apriori} readily follows by noting that, from~\eqref{eq:BNB_1} with~$w_1 = u$ and~$w_2 = \tilde{u}$,
\begin{align*}
M_a \norm{u - \tilde{u}}_U \leq \sup_{v \in V \setminus \{0\}} \frac{a(u; v) - a(\tilde{u}; v)}{\norm{v}_V} = \sup_{v \in V \setminus \{0\}} \frac{\langle f - g, v \rangle_{V', V}}{\norm{v}_V} = \norm{f - g}_{V'} .
\end{align*}

It remains to prove that~\eqref{eq:BNB_1} and~\eqref{eq:BNB_2} are also necessary conditions for ensuring well-posedness of~\eqref{eq:BNB_problem}. The necessity of~\eqref{eq:BNB_1} follows from uniqueness. Indeed, assume that there exists a pair $u_1, u_2 \in U$, $u_1 \neq u_2$, such that
\begin{align*}
\sup_{v \in V \setminus \{0\}} \frac{a(u_1; v) - a(u_2; v)}{\norm{v}_V} = 0 .
\end{align*}
This would imply that~$a(u_1; v) = a(u_2; v)$ for every~$v \in V$, which is in contradiction to uniqueness. The necessity of~\eqref{eq:BNB_2} follows from existence. To see this, assume that there exists some $v_0 \in V \setminus \{0\}$ such that $a(w; v_0) = 0$ for every $w \in U$. By the Hahn-Banach theorem, there exists an $\tilde{f} \in V'$ such that~$\tilde{f}(v_0) \neq 0$, implying
\begin{align*}
0 = a(w; v_0) = \tilde{f}(v_0) \neq 0 ,
\end{align*}
which is a contradiction to the solvability of~\eqref{eq:BNB_problem}. This concludes the proof.
\end{proof}

The second result that we present provides equivalent $\inf$-$\sup$ conditions. It constitutes a nonlinear extension of~\cite[Propositon~A.2]{MelenkSchwab1997}. The result is not essential for the material presented in this paper, but is included nevertheless because it could be of independent interest. 
\begin{theorem} \label{thm:equivalent_infsup}
Let $a \colon U \times V \to \mathbb{R}$ be a semilinear form, such that
\begin{equation}
a(w_1; v) - a(w_2; v) \leq C_a \, \norm{w_1 - w_2}_U \ \norm{v}_V \quad \forall w_1. w_2 \in U, \ \forall v \in V \label{eq:primalBNB_continuity}
\end{equation}
for some constant $C_a > 0$. Moreover, let the map $w \mapsto a(w; \cdot)$ be everywhere Fr\'{e}chet differentiable in $U$, and denote by $a'(q; w, \cdot)$ the corresponding Fr\'{e}chet derivative at~$q \in U$ in the direction $w \in U$. The following statements are equivalent, with identical constant $M_a > 0$.
\begin{enumerate}
\item[(i)] It holds that:
\begin{equation}
\exists M_a > 0 : \quad \inf_{\substack{w_1, w_2 \in U \\ w_1 \neq w_2}} \ \sup_{v \in V \setminus \{0\}} \ \frac{a(w_1; v) - a(w_2; v)}{\norm{w_1 - w_2}_U \norm{v}_V} \geq M_a \, , \label{eq:primalBNB_1}
\end{equation}
\begin{equation}
\sup_{w \in U} a(w; v) > 0 \qquad \forall v \in V \setminus \{0\} \, . \label{eq:primalBNB_2}
\end{equation}
\item[(ii)] For all $q \in U$, it holds that:
\begin{equation}
\inf_{v \in V \setminus \{0\}} \sup_{w \in U \setminus \{0\}} \frac{a'(q; w, v)}{\norm{w}_U \, \norm{v}_V} \geq M_a , \label{eq:dualBNB_1}
\end{equation}
\begin{equation}
\sup_{v \in V} \ a'(q; w, v)> 0 \qquad \forall w \in U \setminus \{0\} \, . \label{eq:dualBNB_2}
\end{equation}
\end{enumerate}
\end{theorem}
\begin{proof}
We first prove that~(i) implies~(ii). For arbitrary fixed $v \in V$, let $J_v \in V'$ such that $\norm{J_v}_{V'} = 1$ and $J_v(v) = \norm{v}_V$. The existence of such a linear functional $J_v \in V'$ follows by application of the Hahn-Banach theorem; see, for example, \cite[p. 5--6]{Zeidler1995}. Then, for any $q \in U$, let $u_{q, v} \in U$ be the solution of
\begin{equation*}
a(u_{q, v}; z) = a(q; z) + \norm{v}_V J_v(z) \qquad \forall z \in V .
\end{equation*}
By Theorem~\ref{thm:nonlinear_infsup} and the premises~\eqref{eq:primalBNB_continuity}--\eqref{eq:primalBNB_2}, we have that the solution $u_{q, v} \in U$ exists and is unique. 
Using~\eqref{eq:primalBNB_1} and recalling that $\norm{J_v}_{V'} = 1$, we deduce the following estimate:
\begin{equation*}
\norm{u_{q, v} - q}_U 
\leq \frac{1}{M_a} \sup_{z \in V \setminus \{0\}} \frac{a(u_{q, v}; z) - a(q; z)}{\norm{z}_V}
= \frac{1}{M_a} \sup_{z \in V \setminus \{0\}} \frac{\norm{v} J_v(z)}{\norm{z}_V}
= \frac{1}{M_a} \norm{v}_V .
\end{equation*}
Hence, we have that
\begin{equation*}
\inf_{v \in V \setminus \{0\}} \frac{a(u_{q, v}; v) - a(q; v)}{\norm{u_{q, v} - q}_U \norm{v}_V}
= \inf_{v \in V \setminus \{0\}} \frac{J_v(v)}{\norm{u_{q, v} - q}_U}
= \inf_{v \in V \setminus \{0\}} \frac{\norm{v}_V}{\norm{u_{q, v} - q}_U}  
\geq M_a .
\end{equation*}
Noting that
\begin{align*}
\inf_{v \in V \setminus \{0\}} \ \sup_{w \in U \setminus \{0\}} \frac{a'(q; w, v)}{\norm{w}_U \norm{v}_V}
\geq \ &
\inf_{v \in V \setminus \{0\}} \ \sup_{w \in U \setminus \{0\}} \ \inf_{t > 0} \frac{a(q + t w; v) - a(q; v)}{t \norm{w}_U \norm{v}_V} 
\\ \geq \ &
\inf_{v \in V \setminus \{0\}} \frac{a(u_{q, v}; v) - a(q; v)}{\norm{u_{q, v} - q}_U \norm{v}_V} ,
\end{align*}
we then obtain~\eqref{eq:dualBNB_1}. To show~\eqref{eq:dualBNB_2}, let
\begin{equation*}
v_0 = \underset{v \in V \setminus \{0\}}{\arg \sup} \left( \inf_{\substack{w_1, w_2 \in U \\ w_1 \neq w_2}} \frac{a(w_1; v) - a(w_2; v)}{\norm{w}_U} \right) .
\end{equation*}
By~\eqref{eq:primalBNB_1}, we have that, for any~$w \in U \setminus \{0\}$,
\begin{equation*}
a'(q; w, v_0) \geq \inf_{t > 0} \frac{a(q + t w; v_0) - a(q; v_0)}{t} \geq \inf_{t > 0} \frac{M_a \norm{t w}_U \norm{v_0}_V}{t} = M_a \norm{w}_U \norm{v_0}_V ,
\end{equation*}
yielding
\begin{equation*}
\sup_{v \in V} a'(q; w, v) 
\geq a'(q; w, v_0) 
\geq M_a \norm{w}_U \norm{v_0}_V > 0 \qquad \forall w \in U \setminus \{0\} .
\end{equation*}
Hence, we have proved that (i) implies (ii).

To prove the reverse implication, consider an arbitrary fixed~$w \in U$, and let $J_w \in U'$ such that $\norm{J_w}_{U'} = 1$ and $J_w(w) = \norm{w}_U$; cf. again \cite[p. 5--6]{Zeidler1995}. Then, given any $q \in U$, let $v_q \in V$ be the solution of
\begin{equation}
a'(q; y, v_q) = \norm{w}_U J_w(y) \qquad \forall y \in U . \label{eq:aux_problem_2}
\end{equation}
By \eqref{eq:primalBNB_continuity}, we have that
\begin{equation*}
a'(q; w, v) \leq C_a \norm{w}_U \norm{v}_V \qquad \forall q, w \in U , \ \forall v \in V .
\end{equation*}
In view of this and the premises~\eqref{eq:dualBNB_1}--\eqref{eq:dualBNB_2}, existence and uniqueness of the solution $v_q \in V$ to the problem~\eqref{eq:aux_problem_2} are asserted by the classical well-posedness result of Banach, Ne\v{c}as and Babu\v{s}ka; cf., for example, \cite[Theorem~1.1]{PietroErn2011}. The following \emph{a priori} estimate is derived:
\begin{equation*}
\norm{v_q}_V 
\leq \frac{1}{M_a} \sup_{y \in U \setminus \{0\}} \frac{a'(q; y, v_q)}{\norm{y}} 
= \frac{1}{M_a} \sup_{y \in U \setminus \{0\}} \frac{\norm{w}_U J_w(y)}{\norm{y}} 
= \frac{1}{M_a} \norm{w}_U .
\end{equation*}
Accordingly, we have that
\begin{equation*}
\inf_{w \in U \setminus\{0\}} \frac{a'(q; w, v_q)}{\norm{w}_U \norm{v_q}_V}
= \inf_{w \in U \setminus \{0\}} \frac{\norm{w}_U J_w(w)}{\norm{w}_U \norm{v_q}_V}
= \inf_{w \in U \setminus \{0\}} \frac{\norm{w}_U}{\norm{v_q}_V}
\geq M_a .
\end{equation*}
By virtue of the mean-value theorem, we then arrive at the inequality
\begin{equation*}
\inf_{\substack{w_1, w_2 \in U \\ w_1 \neq w_2}} \sup_{v \in V \setminus \{0\}} \ \frac{a(w_1; v) - a(w_2; v)}{\norm{w_1 - w_2}_U \norm{v}_V}
\geq \inf_{\substack{w_1, w_2 \in U \\ w_1 \neq w_2}} \sup_{v \in V \setminus \{0\}} \ \inf_{q \in U} \frac{a'(q; w_1 - w_2; v)}{\norm{w_1 - w_2}_U \norm{v}_V}
\geq M_a ,
\end{equation*}
which is~\eqref{eq:primalBNB_1}. Finally, to show~\eqref{eq:primalBNB_2}, let 
\begin{equation*}
w_0 = \underset{w \in U \setminus \{0\}}{\arg \sup} \left( \inf_{q \in U} \inf_{v \in V \setminus \{0\} }\frac{a'(q; w, v)}{\norm{w}_U} \right) .
\end{equation*}
By virtue of the mean-value theorem, we have that
\begin{equation*}
a(w_0; v) \geq \inf_{q \in U} a'(q; w_0, v) \geq M_a \norm{w_0}_U \norm{v}_V > 0 \qquad \forall v \in V \setminus \{0\}
\end{equation*}
This concludes the proof.
\end{proof}


\bibliographystyle{abbrv}
\bibliography{bibliography}

\end{document}